\documentclass{amsart}

\usepackage{amssymb}
\usepackage{amsmath}
\usepackage{esint}
\usepackage{mathtools}
\usepackage{hyperref}
\usepackage{color}

\newcommand{\rd}{{\mathbb{R}^d}}

\newcommand{\R}{{\mathbb {R}}}

\newcommand{\tHs}{{\widetilde H^{s}(\Omega)}}

\newcommand{\Th}{{\mathcal {T}_h}}

\newcommand{\Laps}{(-\Delta)^s}

\newcommand{\SZ}{{\Pi_h}}

\newcommand{\as}[2]{\left( #1 , #2 \right)_s}

\numberwithin{equation}{section} \numberwithin{figure}{section}
{\theoremstyle{remark} \newtheorem{remark}{Remark}}

\newtheorem{proposition}{Proposition}[section]

\newtheorem{theorem}{Theorem}[section]
\newtheorem{corollary}{Corollary}[section]
\newtheorem{lemma}{Lemma}[section]
\newtheorem{example}{Example}[section]

\newcommand{\phii}{\varphi}
\newcommand{\pp}{\partial}
\newcommand{\eps}{\varepsilon}
\newcommand{\x}{\texttt{x}}

\newcommand{\Om}{\Omega}
\newcommand{\na}{\nabla}
\newcommand{\tuh}{\widetilde{u}_h}

\newcommand{\supp}{{\mbox{supp}}}
\newcommand{\dist}{{\mbox{dist}}}

\title{Local energy estimates for the fractional Laplacian}

\author[J.P.~Borthagaray]{Juan Pablo~Borthagaray}
\address[J.P.~Borthagaray]{Departamento de Matem\'atica y
Estad\'istica del Litoral, Universidad de la Rep\'ublica, Salto,
Uruguay}
\email{jpborthagaray@unorte.edu.uy}
\thanks{JPB has been supported in part by NSF grant DMS-1411808 and an AMS-Simons Travel Grant.}

\author[D.~Leykekhman]{Dmitriy~Leykekhman}
\address[D.~Leykekhman]{DUniversity of Connecticut, Storrs, CT 06269, USA}
\email{dmitriy.leykekhman@uconn.edu}

\author[R.H.~Nochetto]{Ricardo H.~Nochetto}
\address[R.H.~Nochetto]{Department of Mathematics and Institute for
Physical Science and Technology, University of Maryland, College
Park, MD 20742, USA}
\email{rhn@math.umd.edu}
\thanks{RHN has been supported in part by NSF grants DMS-1411808 and DMS-1908267.}

\begin{document}

\maketitle

\begin{abstract}
The integral fractional Laplacian of order $s \in (0,1)$ is a nonlocal operator.
It is known that solutions to the Dirichlet problem involving such an operator exhibit an algebraic boundary singularity regardless of the domain regularity. This, in turn, deteriorates the global regularity of solutions and as a result the global convergence rate of the numerical solutions. For finite element discretizations, we derive {\em local} error estimates in the $H^s$-seminorm and show optimal convergence rates in the interior of the domain by only assuming meshes to be shape-regular. These estimates quantify the fact that the reduced approximation error is concentrated near the boundary of the domain. We illustrate our theoretical results with several numerical examples.
\end{abstract}

\section{Introduction}\label{sec:introduction}

In this work we consider finite element discretizations of the problem
\begin{equation} \label{eq:Dirichlet}
\left\lbrace \begin{array}{rl}
\Laps u = f & \mbox{in } \Omega, \\
u = 0 & \mbox{in } \Omega^c := \R^d \setminus \Omega,
\end{array}\right.
\end{equation}
where $\Omega \subset \rd$ is a bounded domain and $\Laps$ is the integral fractional Laplacian of order $s \in (0,1)$,
\begin{equation} \label{eq:def_of_laps}
\Laps u (x) := C_{d,s} \mbox{ p.v. } \int_\rd \frac{u(x) - u(y)}{|x-y|^{d+2s}} \, dy . 
\end{equation}
The normalization constant $C_{d,s} = \frac{2^{2s}s\Gamma\left(s+\frac{d}{2}\right)}{\pi^{d/2}\Gamma(1-s)} $ makes the integral  in \eqref{eq:def_of_laps}, calculated in the principal value sense, coincide with the Fourier definition of $\Laps u$.  It is well understood that, even if the data is smooth (for example, if $\pp\Omega \in C^\infty$ and $f \in C^\infty(\overline\Omega)$), then the unique solution to \eqref{eq:Dirichlet} develops an algebraic singularity near $\partial\Omega$, i.e. a singularity of the form $\operatorname{dist}(x,\partial\Omega)^s$ (cf. Example \ref{ex:nonsmooth}). This is in stark contrast with the classical Laplacian equation.

Nevertheless, in such a case one expects the solution to be locally smooth in $\Omega$, and thus the discretization error to be smaller in the interior of the domain. Our main result (Theorem \ref{thm:local_energy}) is a quantitative estimate of the fact that the finite element error is concentrated around $\pp\Om$.

The fractional Laplacian \eqref{eq:def_of_laps} is a nonlocal operator: computing $\Laps u (x)$  requires the values of $u$ at points arbitrarily far away from $x$. Nonlocality is also reflected in the variational formulation of \eqref{eq:Dirichlet}: the natural space in which the problem is set is the zero-extension fractional Sobolev space $\tHs$, and the norm therein is not subadditive with respect to domain partitions. Furthermore, it is not possible to localize the inner product in $\tHs$, because functions with supports arbitrarily far away from each other may have nonzero $H^s$-inner product. This is also in stark contrast with the local case (i.e., with the inner product in $H^1(\Omega)$), and makes the development of {\em local} estimates for such a {\em nonlocal} problem a more delicate matter, especially in the case of general shape--regular meshes. This is the main purpose of this paper.

In recent years, there has been significant progress in the numerical analysis and implementation of \eqref{eq:Dirichlet} and related fractional-order problems. Finite element discretizations provide naturally the best approximation in the energy norm. A priori convergence rates in the energy norm for approximations using piecewise linear basis functions on either quasi-uniform or graded meshes were derived in \cite{AcosBort2017fractional}; similar results, but regarding convergence in $H^1(\Omega)$ in case $s>\frac{1}{2}$, were obtained in \cite{BoCi19}. The use of adaptive schemes and a posteriori error estimators has been studied in \cite{ainsworth2017aspects, faustmann2019quasi, gimperlein2019space, nochetto2010posteriori, zhao2017adaptive}. A non-conforming discretization, based on a Dunford-Taylor representation was proposed and analyzed in \cite{BoLePa17}. We refer to \cite{BBNOS18,BoLiNo19} for further discussion on these methods. In contrast, the analysis of finite difference schemes typically leads to error estimates in the $L^\infty(\Omega)$-norm under regularity assumptions that cannot be guaranteed in general \cite{duo2018novel, duo2019accurate, huang2014numerical}. 

We learned about \cite{faustmann2020local} after our paper was submitted. {Ref.} \cite{faustmann2020local} also performs a local error analysis for the problem \eqref{eq:Dirichlet}. The local estimates in \cite{faustmann2020local} differ from ours in several respects. The main differences lie in the form of the pollution term, which is expressed in the $H^{s-\frac{1}{2}}$-norm instead of the $L^2$-norm, and that the error estimates are measured in the $H^1$-norm besides the $H^s$-energy norm. The analytical techniques differ as well. While the proof in \cite{faustmann2020local}  is based on the use of the Caffarelli-Silvestre extension, our approach is purely nonlocal and is based on Caccioppoli estimates that are valid for a more general class of kernels \cite{CozziM_2017} and meshes.

The rest of the paper is organized as follows. In Section  \ref{sec:variational_form}, we review the fractional-order spaces and the regularity of solutions to \eqref{eq:Dirichlet} in either standard or weighted Sobolev spaces. In Section \ref {sec:FE}, we describe our finite element discretization, review basic energy based error estimates, and combine such estimates with Aubin-Nitsche techniques to derive novel convergence rates in $L^2$-norm. In Section \ref{sec: Caccioppoli}, we provide a proof of Caccioppoli estimate for the continuous problem. In Section \ref{sec:local_estimates}, which is the central part of the paper, we combine Caccioppoli estimates and  superapproximation techniques, to obtain interior error estimates with respect to $H^s$-seminorms. At the end of this section we show some applications of our interior error estimates. In particular, we discuss the convergence rates of the finite element error in the interior of the domain  with respect to  smoothness of the domain and the right hand side in the case of quasi-uniform and graded meshes. The results are summarized in Tables 1 and 2. Finally,  several numerical examples at the end of the paper illustrate the theoretical results from Section \ref{sec:local_estimates}.

\section{Variational formulation and regularity}\label{sec:variational_form}

In this section, we briefly discuss important features of fractional-order Sobolev spaces that are instrumental for our analysis. Furthermore, we consider regularity properties of the solution to \eqref{eq:Dirichlet} and review some negative results that lead to the use of certain weighted spaces, in which the weight compensates the singular behavior of the gradient of the solution near the boundary of the domain. Having regularity estimates in such weighted spaces at hand, we shall be able to increase the convergence rates by constructing {\em a priori} graded meshes.

\subsection{Sobolev spaces}

Sobolev spaces of order $s \in (0,1)$ provide the natural setting for the variational formulation of \eqref{eq:Dirichlet}. More precisely, we consider $H^s(\rd)$ to be the set of $L^2$-functions $v : \rd \to \R$ such that
\begin{equation} \label{eq:def_Hs_seminorm}
|v|_{H^s(\rd)} :=  \left( \frac{C_{d,s}}2 \int_\rd \int_\rd \frac{|v(x)-v(y)|^2}{|x-y|^{d+2s}} \; dy \; dx \right)^{1/2} < \infty,
\end{equation}
where $C_{d,s}$ is taken as in \eqref{eq:def_of_laps}. Clearly, these are Hilbert spaces; we shall denote by $\as{\cdot}{\cdot}$ the bilinear form that gives rise to the fractional-order seminorms, namely,
\begin{equation} \label{eq:defofinnerprod}
  \as{v}{w} := \frac{C_{d,s}}2 \int_\rd \int_\rd \frac{(v(x)-v(y))(w(x) - w(y))}{|x-y|^{d+2s}} \; dy \; dx .
\end{equation}
For the variational formulation of \eqref{eq:Dirichlet}, we need the zero-extension spaces
\[
\tHs := \{ v \in H^s(\rd) \colon \supp (v) \subset \overline{\Omega} \}, 
\]
for which the form $\as{\cdot}{\cdot}$ becomes an inner product.
Moreover, if $v,w \in \tHs$, then integration in \eqref{eq:defofinnerprod} takes place in $(\rd\times\rd) \setminus(\Omega^c\times\Omega^c)$. We shall denote the $\tHs$-norm by $\|v\|_{\tHs} := \as{v}{v}^{1/2} = |v|_{H^s(\rd)}$, and remark that the $L^2$-norm of $v$ is not needed because a Poincar\'e inequality holds in the zero-extension Sobolev spaces.

Fractional-order Sobolev spaces can be equivalently defined through interpolation of integer-order spaces; remarkably, if one suitably normalizes the standard $K$-functional, then the norm equivalence constants can be taken to be independent of $s$ \cite[Lemma 3.15 and Theorem B.9]{mclean2000strongly}. Although the constant $C_{d,s}$ in \eqref{eq:def_Hs_seminorm} is fundamental in terms of continuity of Sobolev seminorms as $s\to 0 ,1$, we shall omit it whenever $s$ is fixed. For simplicity of notation, throughout this paper we shall adopt the convention $H^0(\Omega) = L^2(\Omega)$.

Let $H^{-s}(\Omega)$ denote the dual space to $\tHs$, and $\langle \cdot , \cdot \rangle$ be their duality pairing. Because of \eqref{eq:defofinnerprod} it follows that if $v \in \tHs$ then $(-\Delta)^s v \in H^{-s}(\Omega)$ and
\begin{equation*}
\label{eq:innerprodisduality}
  \as{v}{w} = \langle (-\Delta)^s v, w \rangle, \quad \forall w \in \tHs.
\end{equation*}
This integration by parts formula motivates the following weak formulation of \eqref{eq:Dirichlet}: given $f \in H^{-s}(\Omega)$, find $u \in \tHs$ such that
\begin{equation} \label{eq:weak_linear}
\as{u}{v} = \langle f, v \rangle \quad \forall v \in \tHs.
\end{equation}
Because this formulation can be cast in the setting of the Lax-Milgram Theorem,
existence and uniqueness of weak solutions, and stability of the solution map $f \mapsto u$, are straightforward.

\subsection{Sobolev regularity} \label{sec:regularity}
Well-posedness of \eqref{eq:weak_linear} in $\tHs$ if $f\in H^{-s}(\Omega)$ is a consequence of the Lax-Milgram Theorem. A subsequent question is what additional regularity does $u$ inherit for smoother $f$. For the sake of finite element analysis, here we shall focus on Sobolev regularity estimates.

By now it is well understood that for smooth domains $\Omega$ and data $f$, solutions to \eqref{eq:Dirichlet} develop an algebraic singular layer of the form (cf. for example \cite{Grubb,RosOtonSerra})
\begin{equation} \label{eq:bdry_behavior}
u(x) \, \dist(x,\pp\Omega)^{-s} = v(x),
\end{equation}
where $v$ is H\"older continuous up to $\partial\Omega$; this limits the global smoothness of solutions. Indeed, if $u$ is locally smooth in $\Omega$ but behaves as \eqref{eq:bdry_behavior}, then one cannot guarantee that $u$ belongs to $H^{s+\frac{1}{2}}(\Omega)$; actually, in general $ u \notin H^{s+\frac{1}{2}}(\Omega)$ (see Example \ref{ex:nonsmooth}).

We now quote a recent result \cite{BoNo19}, that characterizes regularity of solutions in terms of Besov norms. Its proof follows a technique introduced by Savar\'e \cite{Savare98}, that consists in combining the classical Nirenberg difference quotient method with suitably localized translations and exploiting certain convexity properties. We refer to \cite[Section 4]{Savare98} for a definition and basic properties of Besov spaces.

\begin{theorem}[Besov regularity on Lipschitz domains] \label{T:Besov_regularity}
Let $\Omega$ be a bounded Lipschitz domain, $s \in (0,1)$ and $f \in L^2(\Omega)$. Then, there exist constants $C,\zeta$ depending on $\Omega, d$ such that the solution $u$ to \eqref{eq:Dirichlet} belongs to the Besov space $B^{s+\theta}_{2,\infty}(\Omega)$, where $\theta = \frac{1}{2}$ for $\frac{1}{2}<s<1$ and $\theta = s - \epsilon > 0$ for $0 < s \le \frac{1}{2}$, and satisfies the estimates
\begin{equation} \label{eq:Besov_regularity}  
  \| u \|_{B^{s+\theta}_{2,\infty}(\Omega)} \le
  \begin{cases}
    C \big(\frac{1}{2s-1}\big)^{\zeta} \|f\|_{L^2(\Omega)} & \quad  \frac12 < s < 1,
    \\
    C \big(\frac{s}{\eps}\big)^{\zeta} \|f\|_{L^2(\Omega)} & \quad  0 < s \le \frac12.
  \end{cases}
\end{equation}
\end{theorem}
Combining \eqref{eq:Besov_regularity} with the Sobolev embedding $\| u \|_{H^{s+\theta-\eps}(\Omega)} \le \frac{C}{\sqrt\eps} \| u \|_{B^{s+\theta}_{2,\infty}(\Omega)}$ yields
\begin{equation} \label{eq:regularity}
\| u \|_{H^{s+\theta-\eps}(\Omega)} \le\frac{C}{\eps^{\xi}} \|f\|_{L^2(\Omega)} \quad \forall \, 0<\eps < s,
\end{equation}
where $\xi = 1/2$ for $\frac12 < s < 1$ and $\xi=1/2+\zeta$ for $0 < s \le \frac12$ and $C=C(\Omega,d,s)$.

There are two conclusions to be drawn from the previous result. In first place, assuming the domain to be Lipschitz is optimal, in the sense that if $\Omega$ was a $C^\infty$ domain then no further regularity could be inferred. Thus, reentrant corners play no role on the global regularity of solutions: the boundary behavior \eqref{eq:bdry_behavior} dominates any point singularities that could originate from them; we refer to \cite{Gimperlein:19} for further discussion on this point.
In second place, in general the smoothness of the right hand side cannot make solutions any smoother than $\cap_{\eps > 0} \widetilde{H}^{s + \frac{1}{2} - \eps}(\Omega)$. The expression \eqref{eq:bdry_behavior} holds in spite of the smoothness of $f$ near $\pp \Omega$.
We illustrate these two points with a well-known example \cite{Getoor}.

\begin{example}[limited regularity]\label{ex:nonsmooth}
Let $\Omega = B(0,1) \subset \rd$ and $f \equiv 1$. Then, the solution to \eqref{eq:Dirichlet} is
\begin{equation} \label{eq:getoor}
u(x) = \frac{\Gamma(\frac{d}{2})}{2^{2s} \Gamma(\frac{d+2s}{2})\Gamma(1+s)} ( 1- |x|^2)^s_+,
\end{equation}
where $t_+ =\max\{t,0\}$. Therefore, $u \in \cap_{\eps > 0} \widetilde{H}^{s + \frac{1}{2} - \eps}(\Omega)$.
\end{example}

We also point out a limitation in the technique of proof in Theorem \ref{T:Besov_regularity} from \cite{BoNo19} that is related to the example above. Namely, in case $s < \frac{1}{2}$ and $f \in H^r(\Om)$ for some $r > 0$, solutions are expected to be smoother than just $H^{2s}(\Omega)$; however, one cannot derive such higher regularity estimates from Theorem \ref{T:Besov_regularity}. For smooth domains (i.e., $\pp\Om \in C^\infty$), the following estimate holds \cite{VishikEskin}:
\begin{equation} \label{eq:vishik}
f \in H^r(\Omega), \ -s \le r < \frac{1}{2} - s \quad \Rightarrow \quad u \in \widetilde{H}^{2s+r}(\Omega).
\end{equation}

\subsection{Regularity in weighted Sobolev spaces}
By developing a fractional analog of the Krylov boundary Harnack method, Ros-Oton and Serra \cite{RosOtonSerra} obtained a fine characterization of boundary behavior of solutions to \eqref{eq:Dirichlet} and derived H\"older regularity estimates. In order to exploit these estimates and apply them in a finite element analysis, reference \cite{AcosBort2017fractional} introduced certain weighted Sobolev spaces, where the weight is a power of the distance to $\pp\Omega$.
Let
\[
\delta(x) := \dist(x, \pp \Omega), \quad \delta(x,y) := \min \{ \delta(x), \delta(y) \}.
\]
Then, for $k \in \mathbb{N}\cup\{0\}$ and $\gamma \ge 0$, we consider the norm
\begin{equation}
\label{eq:defofWnorm}
  \| v \|_{H^k_\gamma(\Omega)}^2 = \int_\Omega \left( |v(x)|^2 + \sum_{|\beta| \leq k} |\partial^\beta v(x)|^2 \right) \delta(x)^{2\gamma} d x
\end{equation}
and define $H^k_\gamma(\Omega)$ and $\widetilde H^k_\gamma(\Omega)$ as the closures of $C^\infty(\Omega)$ and $C_0^\infty(\Omega)$, respectively, with respect to the norm \eqref{eq:defofWnorm}.

Next, for $t = k + s$, with $k \in \mathbb{N}\cup\{0\}$ and $s \in (0,1)$, and $\gamma \ge 0$, we consider
\begin{equation*} \begin{split}
& \| v \|_{H^{t}_\gamma (\Omega)}^2 := \| v \|_{H^k_\gamma (\Omega)}^2 + 
| v |_{H^{t}_\gamma (\Omega)}^2, \\
& | v |_{H^{t}_\gamma (\Omega)}^2 := 
\int_\Omega \int_\Omega \frac{|\nabla^k v(x)-\nabla^k v(y)|^2}{|x-y|^{d+2s}} \, \delta(x,y)^{2\gamma} \, dy \, dx
\end{split} \end{equation*}
and the associated space $H^t_\gamma (\Omega) :=  \left\{ v \in H^k_\gamma(\Omega) \colon \| v \|_{H^t_\gamma (\Omega)} < \infty \right\} .$

In analogy with the notation for their unweighted counterparts, we define zero-extension weighted Sobolev spaces by
\begin{equation} \label{eq:weighted_sobolev}
\widetilde{H}^{t}_\gamma (\Omega) := \{ v \in H^{t}_\gamma (\rd) : \ v = 0 \mbox{ a.e. in } \Omega^c \}
\end{equation}
with $\| v \|^2_{\widetilde H^{t}_\gamma (\Omega)} := \| v \|^2_{\widetilde H^{k}_\gamma (\Omega)} + | v |^2_{H^{t}_\gamma (\rd)}$. The convenience of using the same weight in both the function and its fractional-order derivatives is discussed in \cite[Section 3]{BoNoSa18}.

We have the following regularity estimate in the scale \eqref{eq:weighted_sobolev} \cite[Proposition 3.12]{AcosBort2017fractional}, \cite[Formula (3.6)]{BBNOS18}.

\begin{theorem}[weighted Sobolev estimate] \label{T:weighted_regularity}
Let $\Omega$ be a bounded, Lipschitz domain satisfying the exterior ball condition, (i.e., there exists $r > 0$ such that for all $x\in \partial \Omega$, there exists $B(y,r)\subset \Omega^c$ satisfying $\overline{B}(y,r)\cap \overline{\Omega}=\{x\}$), $s\in(0,1)$, $f \in C^{\beta}(\overline\Omega)$ for some $\beta \in (0,2-2s)$, $\gamma \ge 0$, $t < \min \{ \beta + 2s, \gamma + s + \frac{1}{2} \}$ and $u$ be the solution of \eqref{eq:weak_linear}. Then, it holds that $u \in \widetilde H^{t}_{\gamma}(\Omega)$ and
\[
\|u\|_{\widetilde H^{t}_{\gamma}(\Omega)} \le \frac{C(\Omega,d,s)}{\sqrt{(\beta + 2s - t) \, (1+2(\gamma + s - t))}} \| f \|_{C^{\beta}(\overline\Omega)}.
\]
\end{theorem}

\begin{remark}[optimal parameters] \label{rmk:weights}
In finite element applications of Theorem \ref{T:weighted_regularity}, discussed in Section \ref{sec:FE}, we will design graded meshes with a grading dictated by $\gamma$. The optimal choice of parameters $t$ and $\gamma$ depends on both the smoothness of the right hand side $f \in C^{\beta}(\overline\Omega)$ and the dimension $d$ of the space. We illustrate this now: let $d\ge 2$, $s < \frac{d}{2(d-1)}$, $\beta = \frac{d}{2(d-1)} - s$, and $\eps > 0$ be sufficiently small, and choose $t = s + \frac{d}{2(d-1)} - \eps d$ and $\gamma = \frac{1}{2(d-1)} - \eps$, to obtain the optimal regularity estimate
\[
\|u\|_{\widetilde H^{t}_{\gamma}(\Omega)} \le \frac{C(\Omega,d,s)}{\eps} \| f \|_{C^{\beta}(\overline\Omega)}.
\]
In contrast, if $s \ge \frac{d}{2(d-1)}$, we set $\beta$ to be any positive number and take $t, \gamma$ as above to arrive at
\[
\|u\|_{\widetilde H^{t}_{\gamma}(\Omega)} \le \frac{C(\Omega,d,s,\beta)}{\sqrt{\eps}} \| f \|_{C^{\beta}(\overline\Omega)}.
\]
\end{remark}

\begin{remark}[exterior ball condition]
Taking into account the results from \cite{Gimperlein:19}, the exterior ball condition could be relaxed. Indeed, such a reference proves that the asymptotic expansion \eqref{eq:bdry_behavior} is valid also for corner singularities, which implies that graded meshes also give rise to optimal convergence rates in that situation. Nevertheless, because the analysis of effects of reentrant corners is beyond the scope of this paper, we leave the exterior ball assumption on $\Omega$.
\end{remark}
  
\section{Finite Element Discretization} \label{sec:FE}
We next consider finite element discretizations of \eqref{eq:weak_linear} by using piecewise linear continuous functions. Let $h_0 > 0$; for $h \in (0, h_0]$, we let $\mathcal{T}_h$ denote a triangulation of $\Om$, i.e., $\mathcal{T}_h = \{T\}$ is a partition of $\Om$ into simplices $T$ of diameter $h_T$. 
We assume the family $\{\Th \}_{h>0}$ to be shape-regular, namely,
\[
\sigma := \sup_{h>0} \max_{T \in \Th} \frac{h_T}{\rho_T} <\infty,
\]
where $h_T = \mbox{diam}(T)$ and $\rho_T $ is the diameter of the largest ball contained in $T$. As usual, the subindex $h$ denotes the element size, $h = \max_{T \in \Th} h_T$; moreover, we take elements to be closed sets.

We shall also need a smooth mesh function $h(x)$, which is locally comparable with the element size. Note that shape-regularity yields $|\na h|\le C(\sigma)$ (cf. \cite[~Lemma~5.1]{NochettoRH_PaoliniM_VerdiC_1991}), and thus
\begin{equation}\label{eq: mesh function}
|h(x)-h(y)|\le C(\sigma) |x-y|, \quad \forall x,y\in \Omega.
\end{equation}

Let $\mathcal{N}_h$ be the set of interior vertices of $\Th$, $N$ be its cardinality , and $\{ \varphi_i \}_{i=1}^N$ the standard piecewise linear Lagrangian basis, with $\phii_i$ associated to the node $\x_i \in \mathcal{N}_h$. With this notation, the set of discrete functions is
\begin{equation*} \label{eq:FE_space}
\mathbb{V}_h :=  \left\{ v \in C_0(\Omega) \colon v = \sum_{i=1}^N v_i \varphi_i \right\}.
\end{equation*}
It is clear that $\mathbb{V}_h \subset \tHs$ for all $s \in (0,1)$ and therefore we have a conforming discretization.

\subsection{Interpolation and inverse estimates}
Fractional-order seminorms are not subadditive with respect to domain decompositions; therefore, some caution must be exercised when localizing them.
With the goal of deriving interpolation estimates, we define the star (or patch) of a set  $A \in \Omega$ by
\[
  S_A := \bigcup \left\{ T \in \Th \colon T \cap A \neq \emptyset \right\}.
\]
Given $T \in \Th$, the star $S_T$ of $T$ is the first ring of $T$ and the star $S_{S_T}$ of $S_T$ is the second ring of $T$.  The star of the node $\x_i \in \mathcal{N}_h$ is $S_i := \mbox{supp}(\varphi_i)$.

We have the following localization estimate for all $v \in H^s(\Omega)$ \cite{Faermann2, Faermann}
\begin{equation} \label{eq:localization}
|v|_{H^s(\Omega)}^2 \leq \sum_{T \in \Th} \left[ \int_T \int_{S_T} \frac{|v (x) - v (y)|^2}{|x-y |^{d+2s}} \; dy \; dx + \frac{C(d,\sigma)}{s h_T^{2s}} \| v \|^2_{L^2(T)} \right]. 
\end{equation}
This inequality shows that to estimate fractional seminorms over $\Omega$, it suffices to compute integrals over the set of patches $\{T \times S_T \}_{T \in \Th}$ plus local zero-order contributions. In addition, if these $L^2$ contributions have vanishing means over elements --as is often the case whenever $v$ is an interpolation error-- a Poincar\'e inequality allows one to estimate them in terms of local $H^s$-seminorms. Thus, one can prove the following local quasi-interpolation estimates (see, for example, \cite{AcosBort2017fractional,BoNoSa18,CiarletJr}).

\begin{proposition}[local interpolation estimates] \label{prop:app_SZ}
Let $T \in \Th$, $s \in (0,1)$, $t \in (s, 2]$, and $\SZ$ be a suitable quasi-interpolation operator. If $v \in H^t (S_{S_T})$, then
\begin{equation} \label{eq:interpolation}
 \int_T \int_{S_T} \frac{|(v-\SZ v) (x) - (v-\SZ v) (y)|^2}{|x-y|^{d+2s}} \, d y \, d x \le C \, h_T^{2(t-s)} |v|_{H^t(S_{S_T})}^2,
\end{equation}
where $C = C(\Omega,d,s,\sigma, t)$.
Moreover, considering the weighted Sobolev scale \eqref{eq:weighted_sobolev}, it holds that for all $v \in H^t_\gamma (S_{S_T})$,
\begin{equation} \label{eq:weighted_interpolation}
 \int_T \int_{S_T} \frac{|(v-\SZ v) (x) - (v-\SZ v) (y)|^2}{|x-y|^{d+2s}} \, d y \, d x \le C
   h_T^{2(t-s-\gamma)} |v|_{H^t_\gamma(S_{S_T})}^2.
\end{equation}
\end{proposition}

For the purpose of this paper, we shall make use of a variant of  \eqref{eq:localization}. Even though the fractional-order norms can be localized, it is clear that the $H^s$-inner product of two arbitrary functions cannot: it suffices to consider two positive functions with supports sufficiently far from each other. The following observation is due to Faermann \cite[Lemma 3.1]{Faermann}. Since we use it extensively, we reproduce it here for completeness.

\begin{lemma}[symmetry]\label{L:symmetry}
  For any $v,w\in L^1(\Omega)$ and $\rho:\R^+\to\R^+$ bounded, there holds
  \[
  \sum_{T\in\Th} \int_T \int_{S_T^c}  v(y) \, w(x) \, \rho(|x-y|) dy dx = 
  \sum_{T\in\Th} \int_T \int_{S_T^c}  v(x) \, w(y) \, \rho(|x-y|) dy dx.
  \]
\end{lemma}
\begin{proof}
We note that, for any two elements $T, T' \in \Th$, it holds $T' \in S_T^c$ if and only if $T \in S_{T'}^c$. Thus, we can write
\[ \begin{aligned}
\sum_{T\in\Th} \int_T \int_{S_T^c}  v(y) \, w(x) \, \rho(|x-y|) dy dx & =
\sum_{T\in\Th} \sum_{T'\in S_T^c} \int_T \int_{T'}  v(y) \, w(x) \, \rho(|x-y|) dy dx \\ 
& = \sum_{T' \in \Th} \sum_{T\in S_{T'}^c} \int_T \int_{T'} v(y) \, w(x) \, \rho(|x-y|) dy dx. 
\end{aligned}
\] 
The proof follows by applying Fubini's Theorem and interchanging the roles of $x$ and $y$.
\end{proof}

\begin{proposition}[equivalent fractional inner product]\label{prop:faermann_inner_products}
Let $v,w \in H^s(\Omega)$. Then, it holds that
\begin{equation*} \begin{split}
(v,w)_{H^s(\Omega)} =  \sum_{T \in \Th} & \bigg[ \int_T \int_{S_T} \frac{(v (x) - v (y))(w(x)-w(y))}{|x-y |^{d+2s}} \, dy \, dx \\
& + 2 \int_T \int_{S_T^c} \frac{v(x) \, (w (x) - w (y))}{|x-y |^{d+2s}} \, dy \, dx \bigg].
\end{split}\end{equation*}
\end{proposition}
\begin{proof}
It suffices to write
\begin{equation*} \begin{split}
(v,w)_{H^s(\Omega)} =  \sum_{T \in \Th} & \bigg[ \int_T \int_{S_T} \frac{(v (x) - v (y))(w(x)-w(y))}{|x-y |^{d+2s}} dy dx  \\ 
& +  \int_T \int_{S_T^c} \frac{(v (x) - v (y))(w(x)-w(y))}{|x-y |^{d+2s}} dy dx \bigg].
\end{split}\end{equation*}
and notice that
\[
 \sum_{T \in \Th} \int_T \int_{S_T^c} \frac{v(x) w(x)}{|x-y |^{d+2s}} \, dy \, dx =  \sum_{T \in \Th} \int_T \int_{S_T^c} \frac{v(y) w(y)}{|x-y |^{d+2s}} \, dy \, dx ,
\]
and
\[
 \sum_{T \in \Th} \int_T \int_{S_T^c} \frac{v(x) w(y)}{|x-y |^{d+2s}} \, dy \, dx =  \sum_{T \in \Th} \int_T \int_{S_T^c} \frac{v(y) w(x)}{|x-y |^{d+2s}} \, dy \, dx
\]
in view of Lemma \ref{L:symmetry} (symmetry) with $\rho(t)=t^{-d-2s} \chi_{[\rho_{min},\infty)}(t)$, where $\rho_{min} = \min_{T \in \Th} \rho_T$ and we recall that $\rho_T$ is the diameter of the largest ball contained in $T$. This completes the proof.
\end{proof}

\begin{remark}[fractional inner product on subdomains]\label{remark: local version}
Proposition \ref{prop:faermann_inner_products} is also valid for any subdomain $D\subset \Om$, i.e.
\[
\begin{aligned}
(v,w)_{H^s(D)} =  \sum_{T \in \Th} &\bigg[ \int_{T\cap D} \int_{S_T\cap D} \frac{(v (x) - v (y))(w(x)-w(y))}{|x-y |^{d+2s}} \, dy \, dx \\ &+ 2 \int_{T\cap D} \int_{S_T^c\cap D} \frac{v(x) \, (w (x) - w (y))}{|x-y |^{d+2s}} \, dy \, dx \bigg].
\end{aligned}
\]
\end{remark}

Next, we write some inverse estimates that we shall use in what follows.
By using standard scaling arguments, one can immediately derive the estimate
\begin{equation}\label{eq: inverse Hs to L2}
\| v_h\|_{H^t(T)} \le C_{inv} h_T^{s-t} \|v_h\|_{H^s(T)}, \quad\forall v_h\in \mathbb{V}_h, \quad 0 \le s\le t \le 1.
\end{equation}
Let $\eta : \Omega \to \R$ be a fixed smooth function. We shall also need the following variant of \eqref{eq: inverse Hs to L2} with $t=1$, whose proof follows immediately because the space $\eta \mathbb{V}_h$ is finite dimensional:
\begin{equation}\label{eq: weighted inverse}
|\eta v_h|_{H^1(S_T)}\le Ch_T^{s-1}|\eta v_h|_{H^s(S_T)}
\quad \forall v_h\in \mathbb{V}_h, \ T \in \Th, \ 0\le s \le 1.
\end{equation}

\subsection{Energy-norm error estimates} \label{sec:energy-norm_estimates}
The discrete counterpart of \eqref{eq:weak_linear} reads: find $u_h \in \mathbb{V}_h$ such that
\begin{equation}\label{eq:weak_linear_discrete}
(u_h, v_h)_s = \langle f, v_h \rangle \quad \forall v_h \in \mathbb{V}_h.
\end{equation}
Subtracting \eqref{eq:weak_linear_discrete} from \eqref{eq:weak_linear} we get Galerkin orthogonality
\begin{equation}\label{eq:orthogonality}
  (u-u_h, v_h)_s = 0 \quad\forall v_h \in \mathbb{V}_h.
\end{equation}
The best approximation property
\begin{equation}\label{eq:best_approximation_linear}
\|u - u_h \|_{\tHs} = \min_{v_h \in \mathbb{V}_h} \|u - v_h \|_{\tHs}
\end{equation}
follows immediately from \eqref{eq:orthogonality}. Consequently, in view of the regularity estimates of $u$ discussed in Section \ref{sec:variational_form}, the only ingredient missing to derive convergence rates in the energy norm is some {\em global} interpolation estimate. Even though the bilinear form $(\cdot, \cdot)_s$ involves integration over $\Omega\times\rd$, it is possible to prove that the corresponding energy norm $\|\cdot\|_{\tHs}$ is bounded in terms of fractional-order norms $\|\cdot\|_{H^s(\Omega)}$ on $\Omega$ by resorting to fractional Hardy inequalities (see \cite{AcosBort2017fractional}). 

Therefore, for {\it quasi-uniform} meshes, if $s\neq \frac{1}{2}$ one can simply combine \eqref{eq:localization} and \eqref{eq:interpolation} with a fractional Hardy inequality \cite[Theorem 1.4.4.4]{Grisvard} to replace $\|\cdot\|_{\widetilde H^s(\Omega)}$ by $\|\cdot\|_{H^s(\Omega)}$ \cite{AcosBort2017fractional,BoNoSa18} and obtain for $t\in(s,1)$
\begin{equation} \label{eq:global_interpolation}
\| v - \SZ v\|_{\widetilde H^s(\Omega)} \le C(\Omega,d,s,\sigma, t) \, h^{t-s} |v|_{H^t(\Omega)} \quad \forall v \in H^t (\Omega).
\end{equation}
In case $s = \frac{1}{2}$, one cannot apply a fractional Hardy inequality. Instead, one may exploit the precise blow-up of the Hardy constant of $H^{\frac{1}{2}+\epsilon}(\Omega)$ as $\epsilon\downarrow 0$ to deduce \cite[\S 3.4]{AcosBort2017fractional}, \cite[Theorem 4.1]{BoNoSa18} for $t\in(\frac12,1)$ and $\eps \in (0,t-\frac12)$
\begin{equation} \label{eq:global_interpolation_1/2}
  \| v - \SZ v\|_{\widetilde H^{\frac12}(\Omega)} \le \frac{C(\Omega,d,\sigma, t)}{\eps} \,
  h^{t-\frac12-\eps} |v|_{H^t(\Omega)} \quad \forall v \in H^t (\Omega).
\end{equation}
Alternatively, one could derive either \eqref{eq:global_interpolation} or \eqref{eq:global_interpolation_1/2} by simply interpolating standard global $L^2$ and $H^1$ estimates. However, if we aim to exploit Theorem \ref{T:weighted_regularity} (weighted Sobolev estimate), then we require a suitable mesh refinement near the boundary of $\Omega$. 
For that purpose, following \cite[Section 8.4]{Grisvard}
we now let the parameter $h$ represent the local mesh size in the interior of $\Omega$, and assume that, besides being shape-regular, the family $\{\Th\}$ is such that there is a number $\mu\ge1$ such that for every $T \in \Th$
 \begin{equation} \label{eq:H}
 h_T \leq C(\sigma) \left\lbrace
 \begin{array}{rl}
   h^\mu, & \mbox{if } T \cap \partial \Omega \neq \emptyset, \\
   h \dist(T,\pp \Omega)^{(\mu-1)/\mu}, & \mbox{if } T \cap \partial \Omega = \emptyset.
 \end{array} \right.
\end{equation} 
This construction yields a total number of degrees of freedom (see \cite{Babuska:79, BoNoSa18})
\begin{equation} \label{eq:dofs}
 N = \dim \mathbb{V}_h \approx
\left\lbrace
 \begin{array}{rl}
    h^{-d}, & \mbox{if } \mu < \frac{d}{d-1}, \\
    h^{-d} | \log h |, & \mbox{if } \mu = \frac{d}{d-1}, \\
    h^{(1-d)\mu}, & \mbox{if } \mu > \frac{d}{d-1}.
  \end{array} \right.
\end{equation}
Thus, if $\mu \le \frac{d}{d-1}$ the interior mesh size $h$ and the dimension $N$ of $\mathbb{V}_h$ satisfy the optimal relation $h \simeq N^{-1/d}$ (up to logarithmic factors if $\mu = \frac{d}{d-1}$). As anticipated in Remark \ref{rmk:weights} (optimal parameters), the weight $\gamma$ in Theorem \ref{T:weighted_regularity} (weighted Sobolev estimate) needs to be related to the parameter $\mu$, which satisfies \eqref{eq:H}. To do so, we combine \eqref{eq:localization} with either \eqref{eq:weighted_interpolation} or \eqref{eq:interpolation}, depending on whether $S_{S_T}$ intersects $\partial\Omega$ or not, to find the relation $\gamma = (t-s) \left( \frac{\mu-1}{\mu} \right)$ for $t\in(s,2]$. If $s\neq \frac{1}{2}$, it suffices to use a fractional Hardy inequality to replace $\|\cdot\|_{\widetilde H^s(\Omega)}$ by $\|\cdot\|_{H^s(\Omega)}$ \cite{AcosBort2017fractional,BoNoSa18} and obtain
\begin{equation} \label{eq:global_weighted_interpolation}
\| v - \SZ v\|_{\widetilde H^s(\Omega)} \le 
\left\lbrace\begin{array}{rl}
C  h^{t-s} |v|_{H^t_\gamma(\Omega)} & \mbox{ if } s \neq \frac{1}{2}, \\ 
\frac{C}{\eps}  h^{t-s-\eps} |v|_{H^t_\gamma(\Omega)} & \mbox{ if } s = \frac{1}{2},
\end{array} \right.
\end{equation}
for all $v \in H^t_\gamma (\Omega)$ with a constant that depends on $\Omega,d,s,\sigma,t$ and $\gamma$. On the other hand, if $s=\frac{1}{2}$, we choose $\gamma = (t-s) \left( \frac{\mu-1}{\mu} \right) - \eps$, where $\eps>0$ is sufficiently small, and exploit the explicit blow-up of the Hardy constant of $H^{\frac{1}{2}+\epsilon}(\Omega)$ as $\epsilon \downarrow 0$, as we did earlier with \eqref{eq:global_interpolation_1/2}, to derive the second estimate in \eqref{eq:global_weighted_interpolation}. We point out that \eqref{eq:global_weighted_interpolation} does not follow by interpolation of global estimates.

We gather the energy error estimates for quasi-uniform and graded meshes in a single theorem.

\begin{theorem}[global energy-norm convergence rates] \label{T:conv_linear}
Let $\Omega \subset \R^d$ be a bounded Lipschitz domain, and $u$ denote the solution to \eqref{eq:weak_linear} and denote by $u_h \in \mathbb{V}_h$ the solution of the discrete problem \eqref{eq:weak_linear_discrete}, computed over a mesh $\Th$ consisting of elements with maximum diameter $h$. If $f \in L^2(\Omega)$, then we have
\begin{equation} \label{eq:conv_Hs}
\|u - u_h \|_{\tHs}  \le C(\Omega,d,s,\sigma) \,  h^\alpha |\log h|^{\kappa} \, \|f\|_{L^2(\Omega)},
\end{equation}
where $\alpha = \min \{s, \frac{1}{2} \}$ and $\kappa = \xi$ if $s \ne \frac{1}{2}$, $\kappa = 1+\xi$ if $s=\frac12$, and $\xi\ge{1/2}$ is the constant in Theorem \ref{T:Besov_regularity}.
Additionally, if $\Omega$ satisfies an exterior ball condition, let $\beta >  0$ be such that
\begin{equation} \label{eq:beta_mu}
\beta \ge \left\lbrace
\begin{array}{rl}
2-2s & \mbox{if } d  = 1, \\
\frac{d}{2(d-1)} - s & \mbox{if } d \ge 2,
\end{array}
\right.
\quad \mbox{and } \quad
\mu = \left\lbrace
\begin{array}{rl}
2-s & \mbox{if } d  = 1, \\
\frac{d}{d-1} & \mbox{if } d \ge 2.
\end{array} \right.
\end{equation}
Then, if $f \in C^{\beta}(\overline{\Omega})$, and the family $\{\Th\}$ satisfies \eqref{eq:H} with $\mu$ as above, we have
\begin{equation} \label{eq:conv_Hs_graded}
\|u - u_h \|_{\tHs}  \le C(\Omega,s,\sigma)  
\left\lbrace
\begin{array}{rl}
h^{2-s} |\log h|^{\kappa-1} \|f\|_{C^{\beta}(\overline{\Omega})} & \mbox{if } d = 1, \\
h^{\frac{d}{2(d-1)}} |\log h|^\kappa \|f\|_{C^{\beta}(\overline{\Omega})} & \mbox{if } d \ge 2,
\end{array} \right.
\end{equation}
where $\kappa = 1$ if $s \ne \frac12$ and $\kappa = 2$ if $s = \frac12$.
In terms of the number of degrees of freedom $N$, the estimate \eqref{eq:conv_Hs_graded} reads
\begin{equation} \label{eq:conv_Hs_graded_N}
\|u - u_h \|_{\tHs}  \le C(\Omega,s,\sigma)  
\left\lbrace
\begin{array}{rl}
N^{-(2-s)} (\log N)^{\kappa-1} \|f\|_{C^{\beta}(\overline{\Omega})} & \mbox{if } d = 1, \\
N^{-\frac{1}{2(d-1)}} (\log N)^{\frac{1}{2(d-1)} + \kappa} \|f\|_{C^{\beta}(\overline{\Omega})} & \mbox{if } d \ge 2.
\end{array} \right.
\end{equation}
\end{theorem}
\begin{proof}  
If $s \neq \frac{1}{2}$, we combine \eqref{eq:best_approximation_linear} and \eqref{eq:global_interpolation} with \eqref{eq:regularity} to obtain
\begin{equation}\label{eq:energy-est}
  \| u - u_h \|_{\tHs} \le C h^{\theta-\eps} |u|_{H^{s+\theta-\eps}(\Omega)} \le C \frac{h^{\theta-\eps}}{\eps^\xi} \|f\|_{L^2(\Omega)},
\end{equation}
where $\theta = \min \{ s - \eps, 1/2 \}$, namely $\theta = \alpha$ if $s > 1/2$ and $\theta = \alpha - \eps$ if $s\le 1/2$.
In case $s=\frac{1}{2}$, instead of \eqref{eq:global_interpolation} we use \eqref{eq:global_interpolation_1/2} with the same $\eps$ as in \eqref{eq:regularity} to get
\begin{equation}\label{eq:energy-est_1/2}
  \| u - u_h \|_{\tHs} \le \frac{C}{\eps} h^{\theta-2\eps} |u|_{H^{s+\theta-\eps}(\Omega)} \le C \frac{h^{\theta-2\eps}}{\eps^{1+\xi}} \|f\|_{L^2(\Omega)}.
\end{equation}
Moreover, coupling \eqref{eq:best_approximation_linear}, the first estimate in \eqref{eq:global_weighted_interpolation} and Theorem \ref{T:weighted_regularity} (weighted Sobolev estimate) with $t = 2 - \eps$ and $\gamma = 2 - s$ if $d = 1$ and $t = s + \frac{d}{2(d-1)} - \eps d$ and $\gamma = \frac{1}{2(d-1)} - \eps$ if $d \ge 2$ yields for $s\neq \frac{1}{2}$
\begin{equation}\label{eq:energy-weight}
 \| u - u_h \|_{\tHs} \le C h^{t-s} |u|_{H^{t}_{\gamma}(\Omega)} \le
\left\lbrace
\begin{array}{rl}
C h^{2-s -\eps} \|f\|_{C^{\beta}(\overline{\Omega})} & \mbox{if } d = 1, \\
 \frac{C}{\eps}  h^{\frac{d}{2(d-1)} - \eps d} \|f\|_{C^{\beta}(\overline{\Omega})} & \mbox{if } d \ge 2;
\end{array} \right.
\end{equation}
analogous estimates hold if $s = \frac{1}{2}$ but with an additional factor $\eps^{-1}{h^{-\eps}}$
according to the second estimate in \eqref{eq:global_weighted_interpolation}.
Upon taking $\eps=|\log h|^{-1}$, we end up with \eqref{eq:conv_Hs} and \eqref{eq:conv_Hs_graded}, as asserted. Inequality \eqref{eq:conv_Hs_graded_N} follows by the choice of $\mu$ and \eqref{eq:dofs}.
\end{proof}

\begin{remark}[exponents of logarithms]
In case $s \ge \frac{d}{2(d-1)}$, which can only happen if $d\ge 3$, the exponents of logarithms in Theorem \ref{T:conv_linear} can actually be reduced by a factor of $\frac{1}{2}$ (see discussion in Remark \ref{rmk:weights}).
\end{remark}

\begin{remark}[optimality]
The convergence rates derived in Theorem \ref{T:conv_linear} are theoretically optimal for shape-regular elements. Nevertheless, because we deal with continuous piecewise linear basis functions, one would expect convergence rate $\frac{-(2-s)}{d}$ with respect to $N$. It is remarkable that such a rate can only be achieved if $d=1$ upon grading meshes according to \eqref{eq:H}. For dimensions $d\ge2$, anisotropic meshes are required in order to obtain optimal convergence rates. This limitation stems from the algebraic singular layer \eqref{eq:bdry_behavior} and becomes more apparent as $d$ increases, but comparison of \eqref{eq:conv_Hs} and \eqref{eq:conv_Hs_graded} shows that in all cases graded meshes improve the convergence rates with respect to $N$.

We also point out that setting the grading parameter to be $\mu > \frac{d}{d-1}$ would lead to a higher rate in \eqref{eq:conv_Hs_graded} in terms of the interior mesh size $h$. However, the resulting rate in \eqref{eq:conv_Hs_graded_N} would be the same as for $\mu = \frac{d}{d-1}$ (up to logarithmic factors) but the finite element matrix would turn out to be worse conditioned.
\end{remark}

\subsection{$L^2$-norm error estimates} \label{sec:conv_L2}
Upon invoking the new regularity estimates of {Theorem} \ref{T:Besov_regularity} for data $f \in L^2(\Omega)$, we now perform a standard Aubin-Nitsche duality argument to derive novel convergence rates in $L^2(\Omega)$. We distinguish between quasi-uniform and graded meshes.

\begin{proposition}[convergence rates in $L^2(\Omega)$ for quasi-uniform meshes] \label{prop:Aubin-Nitsche}
  Let $\Omega$ be a bounded Lipschitz domain. If $f \in L^2(\Omega)$, then for all $0<s<1$ we have
\begin{equation} \label{eq:AN}
\|u - u_h\|_{L^2(\Omega)} \le C h^{2\alpha} | \log h |^{2\kappa} \| f \|_{L^2(\Omega)},
\end{equation}
where $\alpha = \min \{s, \frac{1}{2} \}$, $\kappa = \xi$ if $s \ne \frac12$, $\kappa = 1+\xi$ if $s = \frac12$, and $\xi {\ge 1/2}$ is the constant in \eqref{eq:regularity}.
\end{proposition}
\begin{proof}
Let $e = u - u_h$ be the error, and let $\phi$ be the solution to \eqref{eq:weak_linear} with $e$ instead of the right hand side $f$. Then, the Galerkin orthogonality \eqref{eq:orthogonality} and the Cauchy-Schwarz inequality yield
\[
\| e \|_{L^2(\Omega)}^2 = \as{\phi}{e} = \as{\phi - \Pi_h \phi}{e} \le \|\phi - \Pi_h \phi\|_{\tHs} \|e\|_{\tHs} ,
\]
where $\Pi_h$ is a quasi-interpolation operator satisfying \eqref{eq:global_interpolation} if $s\ne \frac{1}{2}$ or \eqref{eq:global_interpolation_1/2} if $s=\frac{1}{2}$. Combining these estimates with \eqref{eq:regularity}, we deduce for $\eps > 0$ sufficiently small
\begin{equation} \label{eq:approx_phi}
\|\phi - \Pi_h \phi\|_{\tHs} \lesssim
\begin{cases}
  \frac{h^{\theta - \eps}}{\eps^{\xi}} \| e \|_{L^2(\Omega)} &\quad s \ne \frac12
  \\
  \frac{h^{\theta - 2\eps}}{\eps^{1+\xi}} \| e \|_{L^2(\Omega)} &\quad s = \frac12,
\end{cases}
\end{equation}
where $\theta = \min \{ s - \eps, 1/2 \}$, precisely as with \eqref{eq:energy-est} and \eqref{eq:energy-est_1/2}. The latter, together with \eqref{eq:approx_phi}, imply
\[
\| e \|_{L^2(\Omega)} \lesssim
\begin{cases}
  \frac{h^{2(\theta - \eps)}}{\eps^{2\xi}} \| f \|_{L^2(\Omega)} &\quad s \ne \frac12
  \\
  \frac{h^{2(\theta - 2\eps)}}{\eps^{2(1+\xi)}} \| f \|_{L^2(\Omega)} &\quad s = \frac12.
\end{cases}  
\]
Finally, taking $\eps = |\log h|^{-1}$ gives rise to \eqref{eq:AN}.
\end{proof}

In Proposition \ref{prop:Aubin-Nitsche}, the assumption $f\in L^2(\Om)$ is made in order to apply Theorem \ref{T:Besov_regularity} (Besov regularity on Lipschitz domains). Stronger estimates are valid provided $\Omega$ is smooth.

\begin{lemma}[further regularity]\label{L:further-reg}
Let $\pp\Om \in C^\infty$ and $f \in H^r(\Om)$ for some $r \ge -s$. If $\gamma = \min \{ s + r, \frac{1}{2} \}$, $\alpha = \min \{ s , \frac{1}{2} \}$ and $\kappa = 1$ if $s \ne \frac{1}{2}$, $\kappa = 2$ if $s = \frac12$, then there holds
\begin{equation}\label{eq:further-reg} \begin{split}
& \| u - u_h \|_{\tHs} \le C h^{\gamma} | \log h|^{\kappa} \| f \|_{H^r(\Om)}, \\
& \| u - u_h \|_{L^2(\Om)} \le C h^{\alpha + \gamma} | \log h|^{2\kappa} \| f \|_{H^r(\Om)}. 
\end{split} \end{equation}
\end{lemma}
\begin{proof}
Use the regularity result from \cite[Theorem 7.1]{Grubb} (which coincides with \eqref{eq:vishik} if $s<\frac{1}{2}$) in the proofs of Theorem \ref{T:conv_linear} and Proposition \ref{prop:Aubin-Nitsche}.
\end{proof}

As discussed in Sections \ref{sec:regularity} and \ref{sec:energy-norm_estimates}, we obtain a finer characterization of the boundary behavior of solutions by using weighted spaces, and we can take advantage of this by constructing suitably graded meshes. In such a case, the same standard argument as above, but using \eqref{eq:energy-weight} instead of \eqref{eq:energy-est}, leads to the following estimate.

\begin{proposition}[convergence rates in $L^2(\Omega)$ for graded meshes]\label{prop: global_L2_graded}
Let $\Omega \subset \R^d$ be a bounded Lipschitz domain satisfying an exterior ball condition, $f \in C^{\beta}(\overline{\Omega})$ and the family $\{\Th\}$ satisfy \eqref{eq:H}, where $\beta$ and $\mu$ are taken according to \eqref{eq:beta_mu}. Then, there exists a constant $C=C(\Omega,s,\sigma)$ such that
\begin{equation} \label{eq:conv_L2_graded}
\| u - u_h \|_{L^2(\Omega)} \le C 
\left\lbrace
\begin{array}{rl}
h^{2 - s + \alpha} |\log h|^{\kappa-1} \|f\|_{C^{\beta}(\overline{\Omega})} & \mbox{if } d = 1, \\
h^{\frac{d}{2(d-1)} + \alpha} |\log h|^{\kappa} \|f\|_{C^{\beta}(\overline{\Omega})} & \mbox{if } d \ge 2,
\end{array} \right.
\end{equation}
where $\alpha = \min \{s, \frac{1}{2} \}$, $\kappa = \xi+1$ if $s \ne \frac{1}{2}$, $\kappa = \xi + 2$ if $s = \frac12$, and $\xi$ is the constant in \eqref{eq:regularity}. In terms of the number of degrees of freedom $N$, the estimate \eqref{eq:conv_L2_graded} reads
\begin{equation*} \label{eq:conv_L2_graded_N}
\|u - u_h \|_{L^2(\Omega)}  \le C  
\left\lbrace
\begin{array}{rl}
N^{-(2-s+\alpha)} (\log N)^{\kappa-1} \|f\|_{C^{\beta}(\overline{\Omega})} & \! \mbox{if } d = 1, \\
N^{-\frac{\alpha}{d} - \frac{1}{2(d-1)}} (\log N)^{\frac{\alpha}{d} + \frac{1}{2(d-1)} + \kappa} \|f\|_{C^{\beta}(\overline{\Omega})} & \!  \mbox{if } d \ge 2.
\end{array} \right.
\end{equation*}
\end{proposition}

\begin{remark}[sharpness of the $L^2$-estimates]
Combining Galerkin orthogonality \eqref{eq:orthogonality} with \eqref{eq:weak_linear}, and applying the Cauchy-Schwarz inequality, we immediately obtain
\[
\|u-u_h\|_{\tHs}^2 = \as{u-u_h}{u} = \left(u-u_h,f\right)_0 \le \|u-u_h\|_{L^2(\Omega)} \|f\|_{L^2(\Omega)},
\]
from which we deduce that 
\begin{equation} \label{eq:Schatz}
\| u-u_h \|_{L^2(\Omega)} \ge \frac{\|u-u_h\|_{\tHs}^2}{\| f\|_{L^2(\Omega)}}.
\end{equation}
If we knew that the error bound \eqref{eq:conv_Hs} were sharp in the sense that $\| u-u_h \|_{\tHs} \simeq h^\alpha |\log h|^{\kappa} \| f\|_{L^2(\Omega)}$, a reasonable assumption in practice  unless $u \in \mathbb{V}_h$ \cite{lin2014lower}, then we would obtain from \eqref{eq:AN} and \eqref{eq:Schatz}
\begin{equation} \label{eq:pollution_uniform}
\| u-u_h \|_{L^2(\Omega)} \simeq h^{2\alpha} |\log h|^{2\kappa} \| f\|_{L^2(\Omega)}.
\end{equation}

We point out that a similar consideration cannot be made if we inspect weighted estimates. Indeed, let us assume $d\ge 2$ and meshes are graded with parameter $\mu = \frac{d}{d-1}$; similar considerations are valid if the meshes are graded differently. If \eqref{eq:conv_Hs_graded} were sharp, then we could only deduce (up to logarithmic factors)
\begin{equation*} \label{eq:pollution_graded}
h^{\frac{d}{(d-1)}}  \frac{\| f\|_{C^\beta(\overline\Omega)}^2}{\| f\|_{L^2(\Omega)}} \lesssim \| u-u_h \|_{L^2(\Omega)} \lesssim h^{\frac{d}{2(d-1)} + \alpha} \| f\|_{C^\beta(\overline\Omega)},
\end{equation*}
and $\alpha = \min \{ s, \frac{1}{2} \} < \frac{d}{2(d-1)}$.
The issue here is that Theorem \ref{T:weighted_regularity} (weighted Sobolev estimate) does not yield a regularity estimate in terms of $L^2$-norms of the data. Therefore, we still need to use \eqref{eq:approx_phi} which, in turn, is based on the unweighted estimate \eqref{eq:regularity}, a consequence of Theorem \ref{T:Besov_regularity} (Besov regularity on Lipschitz domains).
\end{remark}

\section{Caccioppoli estimate}\label{sec: Caccioppoli}
The following result is well-known for usual harmonic functions. For the fractional Laplacian \eqref{eq:def_of_laps} it can be found, for example, in \cite{CozziM_2017} (see also \cite{Brasco:16,DiCastro:16,Kuusi:15}). We present a proof below, because for our purposes it is crucial to trace the dependence of the constants on the radius $R$ and the exact form of the global term. Moreover, it turns out that the technique of proof will be instrumental in Section \ref{sec:local_estimates}.

\begin{lemma}[Caccioppoli estimate]\label{L:Caccioppoli}
Let $B_{R}$ denote a ball of radius $R$ centered at $x_0\in\Om$.
If $u \in H^s(\R^d)$ is a function satisfying $\int_{B_R^c}\frac{|u(x)|}{|x-x_0|^{d+2s}}dx < \infty$ and $\as{u}{v}=0$ for all $v\in H^s(\R^d)$ supported in $B_R$, then there exists a constant $C$ independent of $R$ such that
\begin{equation} \label{eq:Caccioppoli}
|u|^2_{H^s(B_{R/2})}\le \frac{C}{R^{2s}}\|u\|_{L^2(B_{R})}^2+CR^{d+2s}\left(\int_{B_R^c}\frac{|u(x)|}{|x-x_0|^{d+2s}}dx\right)^2.
\end{equation}
\end{lemma}
\begin{proof}
Let $\eta : \rd \to [0,1]$ be a smooth cut-off function with the following properties:
\begin{subequations} \label{eq: cut-off}
\begin{align}
\eta &\equiv 1 \quad \text{in}\ B_{R/2} \label{property of eta 1}\\
\eta &\equiv 0 \quad \text{in}\ B^c_{3R/4} \label{property of eta 2}\\
|\na\eta| &\le CR^{-1}\label{property of eta 3}.
\end{align}
\end{subequations}
Thus,
\begin{equation*}
  0=\as{u}{\eta^2 u}=\int_\rd \int_\rd \frac{(u(x)-u(y))(\eta^2(x) u(x)-\eta^2(y) u(y))}{|x-y|^{d+2s}}dydx  = I_1 + I_2,
\end{equation*}
where
\begin{align*}
I_1 &:= \int_{B_R}\int_{B_R} \frac{(u(x)-u(y))(\eta^2(x) u(x)-\eta^2(y) u(y))}{|x-y|^{d+2s}}dydx,\\
I_2 &:= 2 \int_{B_R} \int_{B_R^c}\frac{(u(x)-u(y))(\eta^2(x) u(x)-\eta^2(y) u(y))}{|x-y|^{d+2s}}dydx.
\end{align*}
Using the identity 
\begin{equation*}
(u(x)-u(y))(\eta(x)^2 u(x)-\eta(y)^2 u(y)) = [\eta(x) u(x)-\eta(y) u(y)]^2-u(x)u(y)[\eta(x)-\eta(y)]^2,
\end{equation*}
we obtain
$ I_1 =|\eta u|^2_{H^s(B_R)}-I_{11},$
where 
\[
I_{11} = \int_{B_R}\int_{B_R} \frac{u(x)u(y)[\eta(x)-\eta(y)]^2}{|x-y|^{d+2s}}dydx .
\]
In view of of \eqref{property of eta 3}, we have $|\eta(x)-\eta(y)|\le CR^{-1}|x-y|$ and, applying the Cauchy-Schwarz inequality, we deduce 
\begin{align*}
I_{11} &\le  \frac{C}{R^2}\int_{B_R}\int_{B_R} \frac{|u(x)||u(y)|}{|x-y|^{d-2+2s}}dydx
  \\
  & \le
\frac{C}{R^2}\int_{B_R}\int_{B_R} \frac{|u(x)|^2}{|x-y|^{d-2+2s}}dydx\le \frac{C}{R^{2s}}\|u\|^2_{L^2(B_R)},
\end{align*}
because the kernel $|x-y|^{-d+2-2s}$ is integrable on $\{x=y\}$ and using polar coordinates $\rho=|x-y|$ yields
\[
\int_{B_R}\frac{dy}{|x-y|^{d-2+2s}}\le c\int_0^R \rho^{1-2s}d\rho = cR^{2-2s}.
\]
Next, since $\eta$ is supported in $B_{3R/4}$, according to \eqref{property of eta 2}, and bounded by 1, we have
\begin{equation*} \begin{split}
|I_2| & \le 2 \int_{B_R} \int_{B_R^c} \frac{|u(x)-u(y)|\eta^2(x) |u(x)|}{|x-y|^{d+2s}}dydx
\\
& \le  2 \int_{B_{3R/4}} |u(x)|\int_{B^c_R}\frac{|u(x)-u(y)|}{|x-y|^{d+2s}}dydx \le I_{21} + I_{22},
\end{split} \end{equation*}
with
\begin{align*}
  I_{21} := & 2 \int_{B_{3R/4}}\left(|u(x)|^2\int_{B^c_R}\frac{dy}{|x-y|^{d+2s}}\right)dx
  \\
I_{22} := & 2\int_{B_{3R/4}}\left(|u(x)|\int_{B^c_R}\frac{|u(y)|}{|x-y|^{d+2s}}dy\right)dx.
\end{align*}
Using that $\operatorname{dist}(B_{3R/4},B^c_R)=R/4$, and integrating in polar coordinates, we deduce 
\[
\int_{B^c_R}\frac{dy}{|x-y|^{d+2s}}\le C\int_{R/4}^\infty \rho^{-1-2s}d\rho= C R^{-2s} \quad \forall x \in B_{3R/4},
\]
and as a consequence
\[
I_{21}\le \frac{C}{R^{2s}}\|u\|^2_{L^2(B_R)}.
\]
To estimate $I_{22}$, we first observe that for all $x\in B_{3R/4}$ and $y\in B_R^c$,
  we have
\[
R < |y-x_0| \le |x-x_0| + |y-x| \le \frac{3R}{4} + |y-x| \le \frac{3}{4} |y-x_0| + |y-x|
\ \Rightarrow \
\frac14 |y-x_0| \le |y-x|.
\]
Utilizing now the H\"older's inequality, in conjunction with the Young's inequality, yields
\[
\begin{aligned}
I_{22}&\le 2\|u\|_{L^1(B_R)}\sup_{x\in B_{3R/4}}\int_{B^c_R}\frac{|u(y)|}{|x-y|^{d+2s}}dy\le CR^{d/2}\|u\|_{L^2(B_R)}\int_{B^c_R}\frac{|u(y)|}{|y-x_0|^{d+2s}}dy\\
&\le \frac{C}{R^{2s}}\|u\|^2_{L^2(B_R)}+CR^{d+2s}\left(\int_{B^c_R}\frac{|u(y)|}{|y-x_0|^{d+2s}}dy\right)^2.
\end{aligned}
\]
Writing $|\eta u|^2_{H^s(B_R)} = I_{11} - I_2$, and combining the estimates above, we obtain
\[
|\eta u|^2_{H^s(B_R)} \le \frac{C}{R^{2s}}\|u\|^2_{L^2(B_R)}+CR^{d+2s}\left(\int_{B^c_R}\frac{|u(y)|}{|y-x_0|^{d+2s}}dy\right)^2.
\]
The estimate \eqref{eq:Caccioppoli} follows because
\[
|u|^2_{H^s(B_{R/2})}\le |\eta u|^2_{H^s(B_R)},
\]
due to \eqref{property of eta 1}. This concludes the proof.
\end{proof}

\section{Local energy estimates} \label{sec:local_estimates}

In this section we derive error estimates in local $H^s$-seminorms. For that purpose, we first develop a local superapproximation theory in fractional norms and afterwards combine it with the techniques used in the derivation of the Caccioppoli estimate \eqref{eq:Caccioppoli}. 

Here we consider the usual nodal interpolation operator $I_h \colon C_0(\overline{\Om})\to \mathbb{V}_h $, which satisfies for $1\le p\le \infty,\ j\le k\le 2,\ k>\frac{d}{p}$
\begin{equation} \label{eq:interp_Ih}
  |v-I_hv |_{W^{j,p}(T)}\le  Ch^{k-j} |v |_{W^{k,p}(T)}
  \quad\forall \, v\in W^{k,p}(T).
\end{equation}

\subsection{Superapproximation}
Superapproximation is an essential tool in {\it local} energy finite element error estimates \cite{NitscheJA_SchatzAH_1974a}. Below we adapt the ideas from \cite{DemlowA_GuzmanJ_SchatzAH_2011}, which lead to improved superapproximation estimates applicable to a general class of meshes. Similarly to \cite{DemlowA_GuzmanJ_SchatzAH_2011}, we require only shape-regularity.

For an arbitrary $\eta \in C^2 (\overline\Omega)$ and $v_h \in \mathbb{V}_h$, it turns out that the function
\begin{equation}\label{eq:psi}
  \psi := \eta^2v_h- I_h(\eta^2v_h)
\end{equation}      
is smaller than expected in various norms, a property called superapproximation \cite{NitscheJA_SchatzAH_1974a}. To see this, we let $T \in \Th$ be arbitrary and combine \eqref{eq:interp_Ih} with the fact that $v_h$ is linear on $T$, to obtain the following $L^p$-type superapproximation estimate for $\psi$ in \eqref{eq:psi} and any $1\ \le p\le \infty$:
\begin{equation}\label{eq: usual superapproximation}
\begin{aligned}
\|\psi\|_{L^p(T)} + h_T |\psi|_{W^{1,p}(T)} &\le Ch^2_T|\eta^2v_h|_{W^{2,p}(T)} \\
  &\le Ch^2_T\Big(\|\na \eta\|_{L^\infty(T)}\|\na(\eta v_h)\|_{L^p(T)}
  \\
  & + \big(\| \eta\|_{L^\infty(T)}\|\na^2 \eta\|_{L^\infty(T)}+ \|\na \eta\|_{L^\infty(T)}^2\big)\|v_h\|_{L^p(T)}\Big),
\end{aligned}
\end{equation}
where we used that 
\[
\begin{aligned}
\partial^2(\eta^2v_h)&=\partial^2\eta \, (\eta v_h)+2 \, \partial \eta \, \partial(\eta v_h)+\eta \, \partial^2(\eta v_h)\\
&=\partial^2\eta \, (\eta v_h)+2 \, \partial \eta \, \partial(\eta v_h)+\eta\left( \partial^2\eta \, v_h +  {2} \, \partial \eta \, \partial v_h\right)\\
&=2 \, \partial^2\eta \, (\eta v_h)+ 4 \, \partial \eta  \, \partial(\eta v_h) - 2 \, (\partial\eta)^2 \, v_h,
\end{aligned}
\]
with $\partial$ denoting any partial derivative. 
These estimates suffice for second order elliptic problems. However, for fractional problems we need to account for the fact that the $H^s$-norm is nonlocal. We embark on this endeavor now upon first examining stars $S_T$ and next interior balls
\[
B_R := B(x_0, R) \subset \Om,
\quad
h_R := \max_{T\in\Lambda_R} h_T,
\quad
\Lambda_R := \{ T \in \Th \colon T \cap B_R \neq \emptyset \}.
\]
In this setting, $\eta$ is a suitable localization function, namely $\eta\in C^\infty(\Om)$ is the cut-off function of \eqref{eq: cut-off}:
\begin{equation}\label{eq:eta}
  0 \le \eta \le 1, \quad \eta \equiv 1 \quad \text{in}\ B_{R/2}, \quad  \eta \equiv 0 \quad \text{in}\ B^c_{3R/4}, \quad |\nabla^k \eta| \le CR^{-k} \quad (k\ge 1).
\end{equation}
\begin{lemma}[superapproximation in $H^s(S_T)$]\label{L:superapprox-ST}
  Let $T\in\Th$, $0\le s \le 1$, and $\eta$ satisfy \eqref{eq:eta}. For any $v_h\in\mathbb{V}_h$  and $\psi$ given by \eqref{eq:psi}, there is a constant $C$ depending on shape-regularity of $\Th$ such that
  \begin{equation}\label{eq: superapproximation fractional12}
  |\psi|_{H^s(S_T)} \le C \frac{h_T}{R} |\eta v_h|_{H^s(S_T)} + C\frac{h_T^{2-s}}{R^2}
    \|v_h\|_{L^2(S_T)}.
  \end{equation}
\end{lemma}
\begin{proof}
  Since the norms involved in \eqref{eq: usual superapproximation} are local, and the size of $S_T$ is proportional to $h_T$ because $\Th$ is shape-regular, we realize that \eqref{eq: usual superapproximation} is also valid in $S_T$. This leads to the desired estimate for $s=0,1$. For $s \in (0,1)$, we apply space interpolation theory to \eqref{eq: usual superapproximation} over $S_T$ to infer that
\begin{equation}\label{eq: superapproximation fractional}
  |\psi|_{H^s(S_T)} \le C\frac{h_T^{2-s}}{R} \|\na(\eta v_h)\|_{L^2(S_T)} + 
  C \frac{h_T^{2-s}}{R^2} \|v_h\|_{L^2(S_T)}.
\end{equation}
We finally resort to \eqref{eq: weighted inverse}, namely
$\|\na(\eta v_h)\|_{L^2(S_T)} \lesssim h_T^{s-1} |\eta v_h|_{H^s(S_T)}$, to finish the proof.
\end{proof}

\begin{lemma}[superapproximation in $H^s(B_R)$]\label{lemma: superapproximation}
Let $h_R$ satisfy $16 \, h_R \le R$ and let $0\le s\le 1$. For any $v_h\in \mathbb{V}_h$ and $\psi$ given in \eqref{eq:psi}, there exists a constant $C$ depending on shape-regularity of $\Th$ such that
\begin{equation}\label{eq: superapproximation with L2 only}
  |\psi|_{H^s(B_R)} \le C R^{-s} \|v_h\|_{L^2(B_R)}.
\end{equation}
\end{lemma}
\begin{proof}
  If $s = 0, 1$, then the estimate follows immediately from \eqref{eq: usual superapproximation}, the additivity of squares of integer-order $L^2$-norms with respect to domain partitions, the inverse inequality \eqref{eq: weighted inverse} and the fact that $h_R \le R/16$.

For $s \in (0,1)$, we make use of \eqref{eq:localization} to obtain
\begin{align*}
  |\psi|^2_{H^s(B_R)} \le \sum_{T\in \Lambda_R}\left(\int_{T} \int_{S_T}\frac{|\psi(x)-\psi(y)|^2}{|x-y|^{d+2s}}dydx+\frac{C}{h_T^{2s}}\|\psi\|^2_{L^2(T)}\right).
\end{align*}
Let $T\in\Lambda_R$ and $x\in T$ be a generic point. We first point out that if
  $x\in B_{7R/8}^c$ then the vertices $y$ of $T$ satisfy
  $|x_0-y| \ge \frac{7}{8}R - \frac{1}{16}R > \frac{3}{4}R$ and $\psi(x)=0$ according to the definition
  \eqref{eq:psi}. We now let $y \in S_T$ and examine two mutually exclusive cases.
  
If $x\in B_{7R/8}$, then $y$ belongs to a triangle in $\Lambda_R$ because the vertices of
$T$ are at distance $\frac{7}{8}R + \frac{1}{16}R < R$ from $x_0$, whence $|x-y|\le 2h_R\le \frac{1}{8}R$. Therefore
\[
  |x-x_0| \le \frac{7}{8}R \quad\Rightarrow\quad
  |y-x_0| \le |x-x_0| + |y-x| \le R \quad\Rightarrow\quad S_T \subset B_R.
\]
On the other hand, if $x\in B_{7R/8}^c$ and $y \in T'\in\Lambda_R$,
  then $|x-y|\le 2h_R\le \frac{1}{8}R$ and
\[
  |x-x_0| \ge \frac{7}{8}R \quad\Rightarrow\quad
  |y-x_0| \ge |x-x_0| - |y-x| \ge \frac{3}{4}R .
\]
Since $y$ is allowed to be any element vertex on $S_T$, the latter implies that
\begin{equation}\label{eq:7R/8}
\psi\big|_{S_T} \equiv 0 \quad \forall T \in \Lambda_R \setminus \Lambda_{7R/8}.
\end{equation}
We thus realize that the only $T$'s that matter in the sum above are those $T\in \Lambda_{7R/8}$
\begin{align*}
|\psi|^2_{H^s(B_R)} \le \sum_{T\in \Lambda_{7R/8}}\left(|\psi|^2_{H^s(S_T)}+\frac{C}{h_T^{2s}}\|\psi\|^2_{L^2(T)}\right).
\end{align*}
To estimate each term on the right-hand side we exploit the property that $S_T \subset B_R$ for all $T\in\Lambda_{7R/8}$. For the first term, we also employ \eqref{eq: superapproximation fractional}, together with \eqref{eq: weighted inverse} with $s=0$ and \eqref{eq:eta}. For the second term we resort to \eqref{eq: usual superapproximation} for $p=2$ together with \eqref{eq: weighted inverse} for $s=0$. In both cases, we get
\begin{equation*} \begin{split}
\sum_{T\in \Lambda_{7R/8}} \left( |\psi|^2_{H^s(S_T)} +\frac{C}{h_T^{2s}}\|\psi\|^2_{L^2(T)}\right) & \le C \sum_{T\in \Lambda_{7R/8}}  \Big( \frac{h_T^{2-2s}}{R^2}
+ \frac{h_T^{4-2s}}{R^4} \Big) \|v_h\|_{L^2(S_T)}^2 \\
& \le \frac{C}{R^{2s}} \|v_h\|_{L^2(B_R)}^2,
\end{split} \end{equation*}
because $h_T\le \frac{1}{16}R$. The desired estimate follows immediately.
\end{proof}

The proof of Lemma \ref{lemma: superapproximation} (superapproximation in $H^s(B_R)$)
reveals that
\begin{equation}\label{eq:psi=0}
\psi\big|_{B_{7R/8}^c} = 0.
\end{equation}    

\subsection{Local Energy Estimates}
Recall that  the finite element solution to \eqref{eq:weak_linear} satisfies \eqref{eq:weak_linear_discrete}, which gives the Galerkin orthogonality relation \eqref{eq:orthogonality}. In order to localize such relation, given a subdomain $D\subset \Om$, we define $\mathbb{V}_h(D)=\mathbb{V}_h\cap H^1_0(D)$ as the space of continuous piecewise linear functions restricted to $D$ that vanish on $\partial D$. We will derive error estimates for a function $\tuh \in \mathbb{V}_h$ that satisfies the {\it local Galerkin orthogonality} relation
\begin{equation}\label{eq: local Galerkin orthogonality}
 \as{u-\tuh}{v_h}=0,\quad \forall v_h \in \mathbb{V}_h(B_{R}).
\end{equation}

\begin{theorem}[local energy error estimate]\label{thm:local_energy}
Let $u\in\tHs$ and $\tuh \in \mathbb{V}_h$ satisfy \eqref{eq: local Galerkin orthogonality}. If $16 \, h_R \le R$, then there exists a constant $C$ depending on shape regularity such that for any $v_h\in \mathbb{V}_h$,
\[
\begin{aligned}
|u-\tuh|^2_{H^s(B_{R/2})}\le & C |u-v_h|^2_{H^s(B_R)}+\frac{C}{R^{2s}}\|u-v_h\|_{L^2(B_{R})}^2\\
& +CR^{d+2s}\left(\int_{B_R^c}\frac{|u(x)-v_h(x)|}{|x-x_0|^{d+2s}}dx\right)^2
\\ &+\frac{C}{R^{2s}}\|u-\tuh\|_{L^2(B_{R})}^2
+CR^{d+2s}\left(\int_{B_R^c}\frac{|u(x)-\tuh(x)|}{|x-x_0|^{d+2s}}dx\right)^2.
\end{aligned}
\]
\end{theorem}
\begin{proof}
To simplify the notation,  we assume that $B_R = B(0,R)$ is centered at the origin, i.e. we take $x_0=0$. We point out that it is sufficient to establish
\begin{equation}\label{eq: desired simplified}
\begin{aligned}
|\tuh|^2_{H^s(B_{R/2})}\le C |u|^2_{H^s(B_{R})}&+\frac{C}{R^{2s}}\|u\|_{L^2(B_{R})}^2+CR^{d+2s}\left(\int_{B_R^c}\frac{|u(x)|}{|x|^{d+2s}} \, dx \right)^2\\
&+\frac{C}{R^{2s}}\|\tuh\|_{L^2(B_{R})}^2+CR^{d+2s}\left(\int_{B_R^c}\frac{|\tuh(x)|}{|x|^{d+2s}} \, dx \right)^2.
\end{aligned}
\end{equation}
In fact, the assertion would then follow upon writing $u-\tuh=(u-v_h)+(v_h-\tuh)$ and using the fact that the local Galerkin orthogonality \eqref{eq: local Galerkin orthogonality} holds with $u \mapsto u-v_h$ and $\tuh \mapsto  \tuh - v_h$ and the triangle inequality. We argue along the lines of Lemma \ref{L:Caccioppoli} (Caccioppoli estimate). We divide the proof into several steps.

\medskip \noindent
{\it Step 1: Decomposing the $H^s$-seminorm}.
Let $\eta\in C^\infty(\Om)$ be as in \eqref{eq:eta}. Recalling the definition \eqref{eq:psi} of $\psi=\eta^2 \tuh-I_h(\eta^2 \tuh)$, whence $I_h(\eta^2 \tuh)=0$ in $B_{7R/8}^c$ according to the proof of Lemma \ref{lemma: superapproximation}, and using the local Galerkin orthogonality \eqref{eq: local Galerkin orthogonality}, we have
\begin{equation} \label{eq:split_inner_product}
\begin{aligned}
\as{\tuh}{\eta^2 \tuh}&=\as{\tuh}{I_h(\eta^2 \tuh)}+
\as{\tuh}{\psi}\\
&=\as{u}{I_h(\eta^2 \tuh)}+\as{\tuh}{\psi}\\
&=\as{u}{\eta^2 \tuh}-\as{u}{\psi}+\as{\tuh}{\psi}.
\end{aligned}
\end{equation}
In the same fashion as in the proof of Lemma \ref{L:Caccioppoli}, we have
\[
\begin{aligned}
\as{\tuh}{\eta^2 \tuh}= |\eta \tuh |_{H^s(B_R)}^2 &-\int_{B_R} \int_{B_R}\frac{\tuh(x)\tuh(y)[\eta(x)-\eta(y)]^2}{|x-y|^{d+2s}}dydx\\
&+2 \int_{B_R} \int_{B^c_R}\frac{(\tuh(x)-\tuh(y))\eta^2(x) \tuh(x)}{|x-y|^{d+2s}}dydx.
\end{aligned}
\]
Invoking \eqref{eq:split_inner_product} we thus obtain the decomposition
$|\eta \tuh|^2_{H^s(B_R)} = \sum_{k=1}^5 I_k$, where
\begin{equation} \label{eq:bound_seminorm}
\begin{aligned}
I_1 := &\int_{B_R} \int_{B_R}\frac{\tuh(x)\tuh(y)[\eta(x)-\eta(y)]^2}{|x-y|^{d+2s}}dydx,\\
I_2 := &-2\int_{B_R} \int_{B^c_R}\frac{(\tuh(x)-\tuh(y))\eta^2(x) \tuh(x)}{|x-y|^{d+2s}}dydx,\\
I_3 := & \as{u}{\eta^2 \tuh}, \quad
I_4 := -\as{u}{\psi}, \quad
I_5 := \as{\tuh}{\psi}.
\end{aligned}
\end{equation}

\medskip \noindent
{\it Step 2: Bounding $I_1 + I_2$}.
Proceeding exactly as in the proof of Lemma \ref{L:Caccioppoli}, we obtain
\begin{equation*} 
I_1+I_2\le \frac{C}{R^{2s}}\|\tuh\|_{L^2(B_R)}^2+CR^{d+2s}\left(\int_{B_R^c}\frac{|\tuh(x)|}{|x|^{d+2s}}dx\right)^2. 
\end{equation*}

\medskip \noindent 
{\it Step 3: Bounding $I_3$}.
Using the definition of $H^s$-inner product, we write $I_3 = I_{31} + I_{32}$ with
\[
\begin{aligned}
I_{31} :=& \int_{B_R} \int_{B_R}\frac{[u(x)-u(y)] \, [\eta^2(x)\tuh(x)-\eta^2(y)\tuh(y)]}{|x-y|^{d+2s}}dydx,\\
I_{32} :=& 2\int_{B_R} \int_{B^c_R}\frac{[u(x)-u(y)] \, [\eta^2(x)\tuh(x)-\eta^2(y)\tuh(y)]}{|x-y|^{d+2s}}dydx.
\end{aligned}
\]
In light of the identity
\[
\eta^2(x)\tuh(x)-\eta^2(y)\tuh(y)=\eta(x)[\eta(x)\tuh(x)-\eta(y)\tuh(y)]+\eta(y)[\eta(x)-\eta(y)]\tuh(y),
\]
we arrive at
\[
\begin{aligned}
I_{31}=& \int_{B_R} \int_{B_R}\frac{[u(x)-u(y)]\eta(x)[\eta(x)\tuh(x)-\eta(y)\tuh(y)]}{|x-y|^{d+2s}}dydx\\
&+\int_{B_R} \int_{B_R}\frac{[u(x)-u(y)]\eta(y)[\eta(x)-\eta(y)]\tuh(y)}{|x-y|^{d+2s}}dydx\\
&\le | u|_{H^s(B_R)}|\eta \tuh|_{H^s(B_R)}+\frac{C}{R}\int_{B_R} \int_{B_R}\frac{|u(x)-u(y)|\,|\tuh(y)|}{|x-y|^{d-1+2s}}dydx,
\end{aligned}
\]
where in the last step we used that $|\eta| \le 1$ and $|\eta(x)-\eta(y)|\le CR^{-1}|x-y|$ according to \eqref{eq:eta}. Employing the Cauchy-Schwarz inequality, we estimate
\[
\begin{aligned}
\int_{B_R} \int_{B_R} & \frac{|u(x)-u(y)|\, |\tuh(y)|}{|x-y|^{d-1+2s}}dydx \\
&  \le \left(\int_{B_R} \int_{B_R}\frac{|u(x)-u(y)|^2}{|x-y|^{d+2s}}dydx\right)^{\frac12} \
    \left(\int_{B_R}\int_{B_R}\frac{|\tuh(y)|^2}{|x-y|^{d-2+2s}}dydx\right)^{\frac12}\\
&\le CR^{1-s}| u|_{H^s(B_R)}\|\tuh\|_{L^2(B_R)}.
\end{aligned}
\]
In the last step above we used that the kernel $|x-y|^{d-2+2s}$ is integrable at $\{ x = y \}$, and combined Fubini's theorem with integration in polar coordinates, to deduce
\[
\int_{B_R}\frac{dx}{|x-y|^{d-2+2s}}\le C\int_0^{2R} \rho^{d-1-d+2-2s}d \rho = CR^{2-2s}
\quad\forall y \in B_R.
\]
As a result, the Young's inequality yields
\[
I_{31}\le C_\eps| u|^2_{H^s(B_R)}+\eps|\eta \tuh|^2_{H^s(B_R)}+\frac{C}{R^{2s}}\|\tuh\|^2_{L^2(B_R)},
\]
where $\eps > 0$ is a number to be chosen.

To deal with $I_{32}$ we proceed similarly to the estimate of $I_2$ in the proof of Lemma \ref{L:Caccioppoli}. Since $|\eta| \le 1$ and $\eta=0$ on $B^c_{3R/4}$, in view of \eqref{eq:eta}, we thus get
\begin{equation*}
I_{32} \le 
2 \int_{B_{3R/4}} |\tuh(x)|\int_{B^c_R}\frac{|u(x)-u(y)|}{|x-y|^{d+2s}}dydx
\le I_{321} + I_{322}
\end{equation*}
with
\begin{align*}
I_{321} &:= 
2 \int_{B_{3R/4}}\left(|u(x)|\cdot|\tuh(x)|\int_{B^c_R}\frac{dy}{|x-y|^{d+2s}}\right)dx,
\\
I_{322} &:= 2 \int_{B_{3R/4}}\left(|\tuh(x)|\int_{B^c_R}\frac{|u(y)|}{|x-y|^{d+2s}}dy\right)dx.
\end{align*}
Consequently, integrating in polar coordinates
\[
\int_{B^c_R}\frac{dy}{|x-y|^{d+2s}} \le C \int_{R/4}^\infty \rho^{-1-2s} d\rho = \frac{C}{R^{2s}} \quad\forall \, x \in B_{3R/4},
\]
and using the Cauchy-Schwarz inequality, leads to
\[
I_{321}\le C R^{-2s}\|\tuh\|_{L^2(B_R)}\|u\|_{L^2(B_R)}\le C R^{-2s}\|\tuh\|^2_{L^2(B_R)}+CR^{-2s}\|u\|^2_{L^2(B_R)}.
\]
By the H\"older's inequality and the fact that $\frac{1}{4} |y| \le |x-y|$ for all $x \in B_{3R/4}$ and $y \in B_R^c$, we have
\[
\begin{aligned}
I_{322}&\le \|\tuh\|_{L^1(B_R)}\sup_{x\in B_{3R/4}}\int_{B^c_R}\frac{|u(y)|}{|x-y|^{d+2s}}dy\le CR^{d/2}\|\tuh\|_{L^2(B_R)}\int_{B^c_R}\frac{|u(y)|}{|y|^{d+2s}}dy\\
&\le \frac{C}{R^{2s}}\|\tuh\|^2_{L^2(B_R)}+CR^{d+2s}\left(\int_{B^c_R}\frac{|u(y)|}{|y|^{d+2s}}dy\right)^2.
\end{aligned}
\]
Collecting the estimates above, we deduce
\begin{align*}
  I_3 \le \eps|\eta \tuh|^2_{H^s(B_R)} &+ C_\eps|u|^2_{H^s(B_R)}+\frac{C}{R^{2s}}\|u\|^2_{L^2(B_R)} \\
 & +\frac{C}{R^{2s}}\|\tuh\|^2_{L^2(B_R)}+CR^{d+2s}\left(\int_{B^c_R}\frac{|u(y)|}{|y|^{d+2s}}dy\right)^2.
\end{align*}

\medskip \noindent 
{\it Step 4: Bounding $I_4$}.
Using that $\psi=0$ on $B_R^c$ yields the splitting $I_4 = I_{41} + I_{42}$ with
\begin{align*}
  I_{41} &:= -\int_{B_R} \int_{B_R}\frac{[u(x)-u(y)][\psi(x)-\psi(y)]}{|x-y|^{d+2s}}dydx, \\
  I_{42} &:= -2\int_{B_R} \int_{B^c_R}\frac{[u(x)-u(y)]\psi(x)}{|x-y|^{d+2s}}dydx.
\end{align*}
Employing \eqref{eq: superapproximation with L2 only} and the Young's inequality, we obtain
\begin{equation*}
I_{41} \le  |u|_{H^s(B_R)}|\psi|_{H^s(B_R)} \le C|u|^2_{H^s(B_R)}+\frac{C}{R^{2s}}\|\tuh\|^2_{L^2(B_R)}.
\end{equation*}
We handle $I_{42}$ similarly to $I_{32}$, namely use \eqref{eq:psi=0} to write $I_{42} \le I_{421} + I_{422}$ with
\begin{align*}
  I_{421} :=  2 \int_{B_{7R/8}} \! \left(|u(x)| \, |\psi(x)|\int_{B^c_R}\frac{dy}{|x-y|^{d+2s}}\right)dx \!\le\! \frac{C}{R^{2s}}\|u\|^2_{L^2(B_R)} \!+\! \frac{C}{R^{2s}}\|\tuh\|^2_{L^2(B_R)},
\end{align*}
and
\begin{align*}
  I_{422} := & 2\int_{B_{7R/8}}\left(|\psi(x)|\int_{B^c_R}\frac{|u(y)|}{|x-y|^{d+2s}}dy\right)dx \le C \|\psi\|_{L^1(B_R)} \int_{B^c_R}\frac{|u(y)|}{|y|^{d+2s}}dy
  \\ & 
  \le \frac{C}{R^{2s}}\|\tuh\|^2_{L^2(B_R)}+CR^{d+2s}\left(\int_{B^c_R}\frac{|u(y)|}{|y|^{d+2s}}dy\right)^2,
\end{align*}
in view of \eqref{eq: superapproximation with L2 only} with $s=0$, and the fact that
$\frac18 |y| \le |x-y|$ for all $x\in B_{7R/8}$ and $y\in B_R^c$ and argue as in Step 3.
Combining the estimates above, we obtain
\begin{equation*} 
  I_4\le  \, C|u|^2_{H^s(B_R)}+\frac{C}{R^{2s}}\|u\|^2_{L^2(B_R)}
   +\frac{C}{R^{2s}}\|\tuh\|^2_{L^2(B_R)}+CR^{d+2s}\left(\int_{B^c_R}\frac{|u(y)|}{|y|^{d+2s}}dy\right)^2.
\end{equation*}

\noindent
{\it Step 5: Bounding $I_5$}.
We will  treat $I_5$ differently from $I_4$, because it contains $\tuh$ in place of $u$, which causes serious challenges on shape-regular meshes.  Using that $\psi=0$ on $B_R^c$, we split $I_5 = I_{51} + I_{52}$ with
\begin{equation*} \label{eq:split_I5} \begin{split}
& I_{51} := \int_{B_R} \int_{B_R}\frac{[\tuh(x)-\tuh(y)][\psi(x)-\psi(y)]}{|x-y|^{d+2s}}dydx,  \\
& I_{52} := 2 \int_{B_R} \int_{B^c_R}\frac{[\tuh(x)-\tuh(y)]\psi(x)}{|x-y|^{d+2s}}dydx.
\end{split} \end{equation*}
Recalling Remark \ref{remark: local version} (fractional inner product on subdomains), we decompose the integral over $B_R \times B_R$ into sums over $(T\cap B_R)\times (S_T\cap B_R)$ and $(T\cap B_R)\times (S_T^c\cap B_R)$ for $T\in\Lambda_{R}$, and use the fact that $\int_{S_T^c} |x-y|^{-d-2s} dy \le C h_T^{-2s}$ for every  $x \in T$, to end up with
$I_{51} \le I_{511} + I_{512} + I_{513}$ where
\begin{align*}
  I_{511} &:= \sum_{T\in \Lambda_{7R/8}} |\tuh|_{H^s(S_T)}|\psi|_{H^s(S_T)},
   \\
  I_{512} &:= \sum_{T\in \Lambda_{7R/8}} \frac{C}{h_T^{2s}}\int_T |\tuh(x)||\psi(x)|dx,
   \\
  I_{513} &:= 2 \sum_{T\in \Lambda_R} \int_{T\cap B_R}\int_{S_T^c\cap B_R}\frac{|\tuh(x)||\psi(y)|}{|x-y|^{d+2s}}dydx.
\end{align*}
Note that we have used \eqref{eq:7R/8} in the definition of $I_{511}$ and exploited \eqref{eq:psi=0} in the definition of $I_{512}$ to replace $\Lambda_R$ by $\Lambda_{7R/8}$.
We next apply the local inverse inequality \eqref{eq: inverse Hs to L2} in conjunction with the superapproximation estimate \eqref{eq: superapproximation fractional12} to deduce
\begin{align*}
  I_{511} &:= \sum_{T\in \Lambda_{7R/8}} |\tuh|_{H^s(S_T)}|\psi|_{H^s(S_T)}
  \\ &\le C \sum_{T\in \Lambda_{7R/8}} \|\tuh\|_{L^2(S_T)} \Big( \frac{h_T^{1-s}}{R} |\eta\tuh|_{H^s(S_T)} + \frac{h_T^{2-2s}}{R^2} \|\tuh\|_{L^2(S_T)} \Big) \\
&  \le \eps |\eta\tuh|_{H^s(B_R)}^2 + \frac{C_\eps}{R^{2s}} \|\tuh\|_{L^2(B_R)}^2,
\end{align*}
because $16 \, h_T \le R$ and $\sum_{T\in \Lambda_{7R/8}} |v|^2_{H^s(S_T)} \le C(\sigma) |v|^2_{H^s(B_R)}$ for all $v \in H^s(B_R)$, the latter due to the uniformly bounded overlap of stars $S_T$ in the shape-regular mesh $\Th$. The upper bound for $I_{512}$ employs instead the superapproximation estimate \eqref{eq: usual superapproximation} with $p=2$, the inverse inequality \eqref{eq: weighted inverse} and Young's inequality
\begin{align*}
  I_{512} &\le C \sum_{T\in \Lambda_{7R/8}} \frac{1}{h_T^{2s}} \|\tuh\|_{L^2(T)} \|\psi\|_{L^2(T)} \\
   &\le C \sum_{T\in \Lambda_{7R/8}} \|\tuh\|_{L^2(S_T)} \Big( \frac{h_T^{1-s}}{R} |\eta\tuh|_{H^s(S_T)} + \frac{h_T^{2-2s}}{R^2} \|\tuh\|_{L^2(S_T)} \Big) \\
&  \le \eps |\eta\tuh|_{H^s(B_R)}^2 + \frac{C_\eps}{R^{2s}} \|\tuh\|_{L^2(B_R)}^2.
\end{align*}

The remaining term $I_{513}$ is rather tricky and reveals the nonlocal nature of our problem. Manipulating $I_{513}$ is the most delicate and innovative part of the proof relative to the second order case \cite{DemlowA_GuzmanJ_SchatzAH_2011,NitscheJA_SchatzAH_1974a}. To keep notation short, we set
\[
T_R:=T\cap B_R, \quad
S_{T,R}^c:=S^c_T\cap B_{7R/8}, \quad
\Lambda_{T,R}^c := \big\{T'\in\mathcal{T}_h: T'\cap S_{T,R}^c \ne \emptyset  \big\}.
\]
We exploit \eqref{eq:psi=0} to rewrite $I_{513}$ as
\begin{align*}
  I_{513} &= 2\sum_{T\in \Lambda_R} \, \sum_{T' \in \Lambda_{T,R}^c} \int_{T_R}|\tuh(x)|\int_{T'_R}\frac{|\psi(y)|}{|x-y|^{d+2s}}dydx
  \\
  &\le C \sum_{T\in \Lambda_R} \, \sum_{T'\in \Lambda_{T,R}^c} \|\tuh\|_{L^1(T_R)}\|\psi\|_{L^1(T'_R)} \, d(T,T')^{-d-2s}, 
\end{align*}
where $d(T,T')$ denotes the distance between elements $T$ and $T'$. We make use of the superapproximation estimate \eqref{eq: usual superapproximation} with $p=1$ to infer that $I_{513} \le I_{513}^1 + I_{513}^2,$ where
\begin{align*}
  I_{513}^1 &:= C\sum_{T\in \Lambda_R} \, \sum_{T'\in \Lambda_{T,R}^c}\|\tuh\|_{L^1(T_R)}
  \|\tuh\|_{L^1(T'_R)} \, d(T,T')^{-d-2s} \, \frac{h^2_{T'}}{R^2},
  \\
  I_{513}^2 &:= C\sum_{T\in \Lambda_R} \, \sum_{T'\in \Lambda_{T,R}^c}\|\tuh\|_{L^1(T_R)}
  \|\na(\eta \tuh)\|_{L^{1}(T'_R)} \, d(T,T')^{-d-2s} \, \frac{h^2_{T'}}{R}.
\end{align*}  
The first term $I_{513}^1$ is problematic. We rewrite it again in integral form upon invoking the meshsize function $h(y)$, which is locally equivalent to the element meshsize, namely $h(y)\approx h_{T'}$ for all $y\in T'$:
\begin{align*}
  I_{513}^1 \le CR^{-2}\sum_{T\in \Lambda_R}\int_{T_R} \int_{S_{T,R}^c} |\tuh(x)| \frac{h(y)^2|\tuh(y)|}{|x-y|^{d+2s}}dydx \le I_{513}^{11} + I_{513}^{12},
\end{align*}
with
\begin{align*}
& I_{513}^{11} = CR^{-2}\sum_{T\in \Lambda_R} \int_{T_R}|\tuh(x)|^2\int_{S_{T,R}^c} \frac{h(y)^2}{|x-y|^{d+2s}}dydx ,
\\
& I_{513}^{12} =
CR^{-2}\sum_{T\in \Lambda_R} \int_{T_R} \int_{S_{T,R}^c} \frac{h(y)^2|\tuh(y)|^2}{|x-y|^{d+2s}}dydx.
\end{align*}
The first term does not scale correctly unless the meshsize is quasi-uniform, a restriction on $\Th$ that is too severe for us to assume. It is here that we resort to the Lipschitz property \eqref{eq: mesh function} of $h(y)$, valid for shape-regular $\Th$, and integrate in polar coordinates $|x-y|=\rho$, to compute for $x \in T\in \Lambda_R$
\begin{equation*} \begin{split}
\int_{S_{T,R}^c}\frac{h(y)^2}{|x-y|^{d+2s}}dy & \le C \int_{S_T^c\cap B_R}\frac{h(x)^2+ C |x-y|^2}{|x-y|^{d+2s}}dy \\
& \le C\int_{Ch_T}^R \frac{h(x)^2+\rho^2}{\rho^{d+2s}}\rho^{d-1} d\rho\le CR^{2-2s},
\end{split} \end{equation*} 
whence
\begin{equation*}
I_{513}^{11}
 \le \frac{C}{R^{2s}}\sum_{T\in \Lambda_R} \int_{T_R}|\tuh(x)|^2 dx \le \frac{C}{R^{2s}}\|\tuh\|^2_{L^2(B_R)}.
\end{equation*} 
On the other hand, resorting to Lemma \ref{L:symmetry} (symmetry), we have
\begin{align*}
I_{513}^{12} & \le C R^{-2}
\sum_{T\in\Th} \int_T \int_{S_T^c} \frac{h(y)^2|\tuh(y)|^2\chi_{B_R}(y) \chi_{B_R}(x)}{|x-y|^{d+2s}} \, dy dx
\\
& = C R^{-2} \sum_{T\in\Th} \int_T \int_{S_T^c} \frac{h(x)^2|\tuh(x)|^2\chi_{B_R}(x) \chi_{B_R}(y)}{|x-y|^{d+2s}} \, dy dx
\\
& = C R^{-2} \sum_{T\in\Th} \int_T h(x)^2|\tuh(x)|^2\chi_{B_R}(x) \int_{S_T^c} \frac{\chi_{B_R}(y)}{|x-y|^{d+2s}} \, dy dx,
\end{align*}
where $\chi_{B_R}$ denotes the characteristic function of $B_R$. Since
\begin{equation*}
\int_{S_T^c} \frac{\chi_{B_R}(y)}{|x-y|^{d+2s}} \, dy \le C \int_{Ch_T}^R \rho^{-1-2s} d\rho \le C h_T^{-2s} \quad \forall \, x\in T,
\end{equation*}
$h(x) \approx h_T$ for all $x\in T$ and $16 \, h_T \le R$, we see that
\begin{align*}
I_{513}^{12} \le C R^{-2}  
\sum_{T\in \Th} h_T^{2-2s} \int_T \chi_{B_R} (x)|\tuh(x)|^2 dx 
\le CR^{-2s}\|\tuh\|^2_{L^2(B_R)}.
\end{align*}
Collecting the preceding estimates for $I_{513}^1$, we realize that
\[
I_{513}^1 \le CR^{-2s}\|\tuh\|^2_{L^2(B_R)}.
\]
We handle $I_{513}^2$ similarly to $I_{513}^1$, namely
\begin{align*}
I_{513}^2 
&\le CR^{-1}\sum_{T\in \Lambda_R}\int_{T_R} \int_{S_{T,R}^c} |\tuh(x)| \frac{h(y)^2|\na(\eta \tuh)(y)|}{|x-y|^{d+2s}}dydx\\
&\le C_\eps R^{-2}\sum_{T\in \Lambda_R}\int_{T_R}\int_{S_{T,R}^c}\frac{|\tuh(x)|^2 h(y)^2}{|x-y|^{d+2s}}dydx
\\ & \quad + C\eps \sum_{T\in \Lambda_R}\int_{T_R}\int_{S_{T,R}^c} \frac{h(y)^2|\na(\eta \tuh)(y)|^2}{|x-y|^{d+2s}}dydx\\
&\le C_\eps R^{-2s}\|\tuh\|^2_{L^2(B_R)}+C\eps \sum_{T\in \Lambda_R}\int_{T_R}\int_{S_{T,R}^c} \frac{h(y)^2|\na(\eta \tuh)(y)|^2}{|x-y|^{d+2s}}dydx,
\end{align*}
since the first term is identical to $I_{513}^{11}$. For the other term in the right hand side, we proceed exactly as with $I_{513}^{12}$, thereby exploiting again Lemma \ref{L:symmetry} (symmetry) and combining it with the inverse-type estimate \eqref{eq: weighted inverse}, to obtain
\begin{align*}
\sum_{T\in \Lambda_R} \int_{T_R} \int_{S_{T,R}^c}\frac{h(y)^2|\na(\eta \tuh)(y)|^2}{|x-y|^{d+2s}}dydx  \le C \sum_{T\in \Th} h_T^{2-2s}|\eta \tuh|^2_{H^1(T)}  \le C \, |\eta \tuh|^2_{H^s(B_R)}.
\end{align*}
Combining the estimates for $I_{511}$, $I_{512}$, $I_{513}$ we deduce that
\[
I_{51}\le C_\eps R^{-2s}\|\tuh\|^2_{L^2(B_R)}+ C \, \eps \, |\eta \tuh|^2_{H^s(B_R)}.
\]
It only remains to bound $I_{52}$, which is exactly the same as $I_{42}$ but with $u$ replaced by $\tuh$. Hence, proceeding similarly to the estimate for $I_{42}$, we readily arrive at
\begin{equation*}
I_{52} \le \frac{C}{R^{2s}}\|\tuh\|^2_{L^2(B_R)}+CR^{d+2s}\left(\int_{B^c_R}\frac{|\tuh(y)|}{|y|^{d+2s}}dy\right)^2.
\end{equation*}
This together with the previous estimate yields
\begin{equation*}
I_5\le C \eps |\eta \tuh|^2_{H^s(B_R)}+\frac{C_\eps}{R^{2s}}\|\tuh\|^2_{L^2(B_R)}+CR^{d+2s}\left(\int_{B^c_R}\frac{|\tuh(y)|}{|y|^{d+2s}}dy\right)^2.
\end{equation*}

\medskip \noindent
{\it Step 6: Conclusion}.
Inserting the bounds proved in Steps 2 through 5 for $I_i, 1\le i \le 5$ into \eqref{eq:bound_seminorm}, we deduce that
\[
\begin{aligned}
|\eta \tuh|^2_{H^s(B_R)}&\le C_\eps |u|^2_{H^s(B_{R})}+\frac{C}{R^{2s}}\|u\|_{L^2(B_{R})}^2+CR^{d+2s} \left(\int_{B_R^c}\frac{|u(x)|}{|x|^{d+2s}} \, dx \right)^2\\
& + C \eps |\eta \tuh|^2_{H^s(B_R)}+\frac{C_\eps}{R^{2s}}\|\tuh\|_{L^2(B_{R})}^2+CR^{d+2s} \left(\int_{B_R^c}\frac{|\tuh(x)|}{|x|^{d+2s}} \, dx\right)^2,
\end{aligned}
\]
for all $\eps > 0$.
We now set $\eps$ to be such that the factor multiplying $|\eta \tuh|^2_{H^s(B_R)}$ in the right hand side equals $\frac{1}{2}$ and kick that term back to the left hand side.
This finally implies the estimate \eqref{eq: desired simplified} because  $|\tuh|^2_{H^s(B_{R/2})}\le |\eta \tuh|^2_{H^s(B_R)}$.
\end{proof}

We can derive explicit local $H^s$-convergence rates by combining Theorem \ref{thm:local_energy} with the convergence estimates from Section \ref{sec:FE}. We explore this next.

\subsection{Applications to interior error estimates}\label{S:applications}

Theorem \ref{thm:local_energy} (local energy error estimate) gives us new ways to examine the behavior of the numerical error and, more importantly, check the sharpness of known estimates. Bounding the low order terms in Theorem \ref{thm:local_energy} by global $L^2$-terms, we get the following immediate consequence of Theorem \ref{thm:local_energy}.

\begin{corollary}[local error estimate]\label{cor: simplified}
Let $u\in H^s(\Omega)$ be the solution of \eqref{eq:weak_linear} and $u_h$ be the finite element solution of \eqref{eq:weak_linear_discrete}. Then there is a constant $C$ depending on shape regularity such that
\begin{align*}
|u-u_h|_{H^s(B_{R/2})}\le C \inf_{v_h\in \mathbb{V}_h} \Big(|u-v_h|_{H^s(B_R)}&+\frac{1}{R^{s}}\|u-v_h\|_{L^2(\Omega)} \Big) + \frac{C}{R^{s}}\|u - u_h\|_{L^2(\Omega)}.
\end{align*}
\end{corollary}
\begin{proof}
We apply Theorem \ref{thm:local_energy} to $u$ and $u_h$, which clearly satisfies the local Galerkin orthogonality condition \eqref{eq: local Galerkin orthogonality}. The proof then follows from the Cauchy-Schwarz inequality and integration in polar coordinates
\[
  \begin{aligned}
  R^{d+2s}\left(\int_{B_R^c}\frac{w(x)}{|x-x_0|^{d+2s}}dx\right)^2
  &\le R^{d+2s}\|w\|^2_{L^2(\Omega)}\int_{B_R^c}\frac{1}{|x-x_0|^{2d+4s}} \, dx\\
  & \le CR^{d+2s}\|w\|^2_{L^2(\Omega)}\int_{R}^\infty\frac{\rho^{d-1}}{\rho^{2d+4s}} \, d\rho
  =\frac{C}{R^{2s}}\|w\|^2_{L^2(\Omega)},
  \end{aligned}
\]
    for $w = |u-v_h|$ and $w = |u-u_h|$. This concludes the proof.
\end{proof}

Since $\|u - \Pi_h u\|_{L^2(\Omega)} \le C \|u - u_h\|_{L^2(\Omega)}$ generically,
  Corollary \ref{cor: simplified} shows that the interior $H^s$-error consists of a local approximation error in the $H^s$-norm and a global $L^2$-Galerkin error that accounts for {\it pollution} from the rest of the domain. We observe that this estimate is similar to local estimates for second order elliptic problems \cite{DemlowA_GuzmanJ_SchatzAH_2011,NitscheJA_SchatzAH_1974a}, except that the $L^2$-terms are now global. This is a mild manifestation of the nonlocal nature of \eqref{eq:Dirichlet}. We examine below the extreme cases of quasi-uniform and graded meshes.

Since the polynomial degree of $\mathbb{V}_h$ is $1$, no error estimate can be of order larger than $2$ and exploit regularity of $u$ beyond $H^2$ regardless of mesh structure. 
With this in mind, we let $f\in H^r(\Om)$ for $0\le r \le 2-2s$, and assume it leads to the {\it local} $H^{2s+r}_{\text{loc}}$-regularity of $u$ and the local approximation error
\begin{equation} \label{eq:interior_regularity}
\inf_{v_h\in \mathbb{V}_h} |u-v_h|_{H^s(B_R)} \le Ch^{s+r}\|f\|_{H^r(\Omega)}.
\end{equation}
We remark that this regularity assumption is plausible and known to be true for $r \le 1-s$ (see for example \cite{faustmann2020local}, and \cite{BiWaZu17} for a proof in the case $r=0$) and that if $\Omega$ is smooth and $f \in W^{r,p}(\Omega)$ for some $p>d/s$ then $u \in W^{r+2s, p}_{loc}(\Omega)$ \cite{Grubb}.
In order to compare with the global $H^s$-estimate of Theorem \ref{T:conv_linear} (global energy-norm convergence rates), we consider below the best scenario of maximal interior regularity, namely the case where the rate $s+r$ in \eqref{eq:interior_regularity} is sufficiently large $s+r\ge 1$, so that the local $H^s$-rate is dictated by the global $L^2$-error.

\medskip\noindent
{\bf Quasi-uniform meshes.}
Combining \eqref{eq:interior_regularity} with the estimates of Proposition \ref{prop:Aubin-Nitsche} (convergence rates in $L^2(\Omega)$ for quasi-uniform meshes) and Lemma \ref{L:further-reg} (further regularity) of Section \ref{sec:conv_L2}, we obtain
\begin{equation*} 
|u-u_h|_{H^s(B_{R/2})}\le Ch^{s+r}\|f\|_{H^r(\Omega)}+\begin{cases}
Ch^{2\alpha} |\log h|^{2\kappa} \| f \|_{L^2(\Om)} &  \text{ for  $\Omega$ Lipschitz} \\
Ch^{\alpha+\gamma} |\log h|^{2\kappa} \| f \|_{H^r(\Om)} & \text{ for  $\Omega$ smooth},
\end{cases}
\end{equation*}
where $\alpha = \min \{ s , \frac{1}{2} \}$, $\gamma = \min \{ s + r, \frac{1}{2} \}$ and 
if $\Omega$ is Lipschitz then $\kappa = \xi$ for $s \ne \frac12$ and $\kappa = \xi + 1$ for $s = \frac12$ ($\xi$ is the constant in \eqref{eq:regularity}), whereas if $\Omega$ is smooth then $\kappa = 1$ for $s \neq\frac12$ and $\kappa = 2$ for $s = \frac{1}{2}$.
We summarize these estimates in Table \ref{tab:hresult} (up to logarithmic factors){; we remark that the rates therein for Lipschitz domains do not require the exterior ball condition}. Compared with Theorem \ref{T:conv_linear} (global energy-norm convergence rates)
\begin{equation}\label{eq:global-rate}
\|u-u_h\|_{\widetilde{H}^s(\Omega)} \le C h^{\min\{s,\frac{1}{2}\}} |\log h|^{\kappa} \, \|f\|_{L^2(\Omega)},
\end{equation}
we see that all interior $H^s$-rates of Table \ref{tab:hresult} are improvements over the global rate of \eqref{eq:global-rate}. For a more regular right hand side $f \in H^r(\Om)$ with $s + r \ge \frac{1}{2}$ and in smooth domains, we observe an improvement over the global rate dictated by Lemma \ref{L:further-reg}, 
\[
\| u - u_h \|_{\tHs} \le C h^{\frac{1}{2}} |\log h|^{\kappa} \| f \|_{H^r(\Om)}. 
\]

\begin{table}[h]
\centering
\begin{tabular}{|c||c|c||c|c|}
\hline 
& \multicolumn{2}{|c||}{Local rates} & \multicolumn{2}{|c|}{Global rates} \\ \cline{2-5}
& $\Omega$-smooth & $\Omega$-Lipschitz & $\Omega$-smooth & $\Omega$-Lipschitz \\ \hline
$s \le \frac{1}{2}$ & $h^{s+\frac{1}{2}}$ & $h^{2s}$  & $h^{\frac{1}{2}}$ & $h^{s}$ \\
$s > \frac{1}{2}$ & $h$ & $h$  & $h^{\frac{1}{2}}$ & $h^{\frac{1}{2}}$ \\  
\hline
\end{tabular}
\vskip0.3cm
\caption{Comparison of convergence rates (up to logarithmic factors) between interior $|u-u_h|_{H^s(B_{R/2})}$ and global $|u-u_h|_{H^s(\Omega)}$ error estimates on quasi-uniform meshes for $f \in H^r(\Omega)$ with $s+r \ge 1$. The interior estimates exhibit an improvement $h^{\min\{s,1/2\}}$ regardless of the regularity of $\Omega$.}
\label{tab:hresult}
\end{table}

\noindent
{\bf Graded meshes.}
Section \ref{sec:FE} shows that graded meshes satisfying \eqref{eq:H} are able to compensate for the singular boundary layer for Lipschitz domains satisfying the exterior ball condition and smooth right-hand sides. Even though the next discussion is valid for any dimension $d$, for the sake of clarity and because our numerical experiments in Section \ref{sec:numerical} are carried out for $d=2$, we shall focus on this case. Moreover, we assume $s\neq \frac{1}{2}$, for otherwise additional logarithmic factors arise in our estimates below. We set $\mu = 2$ and $\beta = 1-s$ in Theorem \ref{T:conv_linear} (global energy-norm convergence rates) and Proposition \ref{prop: global_L2_graded} (convergence rates in $L^2(\Omega)$ for graded meshes) to establish the global rates of convergence in $\widetilde{H}^s(\Omega)$ and $L^2(\Omega)$
\begin{align}
\label{eq:global-rate-graded-Hs}    
&  \|u-u_h\|_{\widetilde{H}^s(\Omega)} \le C h |\log h| \, \|f\|_{C^{1-s}(\overline{\Omega})},
\\
\label{eq:global-rate-graded-L2} 
&  \|u-u_h\|_{L^2(\Omega)} \le C h^{\min\{1+s,3/2\}} |\log h|^{\xi+1} \, \|f\|_{C^{1-s}(\overline{\Omega})};
\end{align}
$\xi$ is the constant in \eqref{eq:regularity}.
In contrast, Theorem \ref{thm:local_energy} (local energy error estimate) in conjunction with \eqref{eq:global-rate-graded-L2} for $f \in C^{1-s}(\overline\Omega)\cap H^r(\Omega)$, $0\le r \le 2-2s$, gives the local $H^s$-estimate
\begin{equation*} 
  |u-u_h|_{H^s(B_{R/2})}\le Ch^{s+r}\|f\|_{H^r(\Omega)}+Ch^{\min\{1+s,3/2\}}
  |\log h|^{\xi+1} \| f \|_{C^{1-s}(\overline\Omega)}.
\end{equation*}
The condition $r \le 2-2s$ above is related to the use of piecewise linear finite elements. We now assume that $s+r = 2-s$ to write
\begin{equation*}\label{eq:interior-graded}
|u-u_h|_{H^s(B_{R/2})}\le C h^{\min\{1+s,2-s\}} |\log h|^{\xi+1} \, \Big(\| f \|_{C^{1-s}(\overline\Omega)} + \|f\|_{H^{2-2s}(\Omega)} \Big).
\end{equation*}
Comparing with the global $H^s$-error estimate in \eqref{eq:global-rate-graded-Hs}, we thus see an overall improvement rate $h^{\min\{s,1-s\}}$. We summarize these results in Table \ref{tab:hresult_graded}.   
 
\begin{table}[h]
\centering
\begin{tabular}{|c||c||c|}
\hline 
& \multicolumn{2}{|c|}{$\Omega$-smooth or Lipschitz e.b.c.} \\ \cline{2-3}
& { Local rates } & {Global rates} \\ \hline
$s \le \frac{1}{2}$ & $h^{s+1}$ & $h$ \\
$s > \frac{1}{2}$ & $h^{2-s}$ & $h$ \\  
\hline
\end{tabular}
\vskip0.3cm
\caption{Comparison of order of convergence (up to logarithmic factors) between interior $|u-u_h|_{H^s(B_{R/2})}$ and global $|u-u_h|_{H^s(\Omega)}$ error estimates on graded meshes with parameter $\mu=2$ for $f \in H^{2-2s}(\Omega)\cap C^{1-s}(\overline\Omega)$. The interior estimates exhibit an improvement rate $h^{\min\{s,1-s\}}$ for $\Omega$ either smooth or Lipschitz with an exterior ball condition (e.b.c.).}
\label{tab:hresult_graded}
\end{table}

We conclude with a comparison between local error rates on quasi-uniform and graded meshes for smooth data (domain and right-hand side). Tables \ref{tab:hresult} and \ref{tab:hresult_graded} show  that graded meshes yield an improvement of order $h^{\frac{1}{2}}$ for all $s \le \frac{1}{2}$, whereas the improvement is of order $h^{1-s}$ for $s > \frac{1}{2}$. Therefore, such an improvement is valid for all $0<s<1$ but becomes less significant in the limit $s \to 1$ of classical diffusion.

\section{Numerical experiments} \label{sec:numerical}

In this section we present some numerical experiments in a two-dimensional domain that illustrate the sharpness of our theoretical estimates. These experiments were performed with the aid of the code documented in \cite{ABB}; we also refer to \cite{ABB} for details on the implementation. Some discussion about the construction of graded meshes satisfying \eqref{eq:H} can be found in \cite{AcosBort2017fractional}.

In all of the experiments below we set $\Omega = B(0,1) \subset \mathbb{R}^2$ and $f \equiv 1$, so that we have an explicit solution at hand (cf. Example \ref{ex:nonsmooth}). This corresponds to smooth data (both domain and right-hand side) and the discussion of Section \ref{S:applications} applies. We computed errors with respect to the dimension $N$ of the finite element spaces $\mathbb{V}_h$ because $N = \#\mbox{Dofs}$ is a measure of complexity. In view of \eqref{eq:dofs} with $\mu = 2$, we always have the relation $N \approx h^{-2}$ for both quasi-uniform and graded meshes, the latter up to logarithmic terms. Therefore, the rates of convergence of Section \ref{S:applications} can be expressed in terms of $N$ as follows
\begin{equation}\label{eq:h-N} 
h^\beta \approx N^{-\beta/2},
\end{equation}
for appropriate exponents $\beta>0$.
We next explore computationally our error estimates in Section \ref{sec:conv_L2} for both the global $L^2$-norm and local $H^s$-seminorm.

\subsection{Global $L^2$-norm error estimates}
We start with quasi-uniform meshes and $s = 0.5, 0.6, 0.7, 0.8, 0.9$. Our findings are summarized in Figure \ref{fig:errsL2_uniform}: in all cases, we see good agreement with the linear convergence rate $\beta=1$ predicted by Proposition \ref{prop:Aubin-Nitsche} for $s\ge 1/2$, or equivalently $N^{-1/2}$ according to \eqref{eq:h-N}. Since the exact solution satisfies $u \in \cap_{\eps>0} \widetilde{H}^{s+1/2-\eps}(\Omega)$, we infer that the $L^2$-interpolation error obeys the inequality $\|u-I_h u\|_{L^2(\Omega)}\le C h^{s+1/2} |\log h|$. Interestingly, the finite element error $\|u-u_h\|_{L^2(\Omega)}\le C h |\log h|^2$ is of lower order for $s>1/2$, which turns out to be consistent with \eqref{eq:pollution_uniform}.

\begin{figure}[htbp]
\begin{center}
  \includegraphics[width=0.8\linewidth]{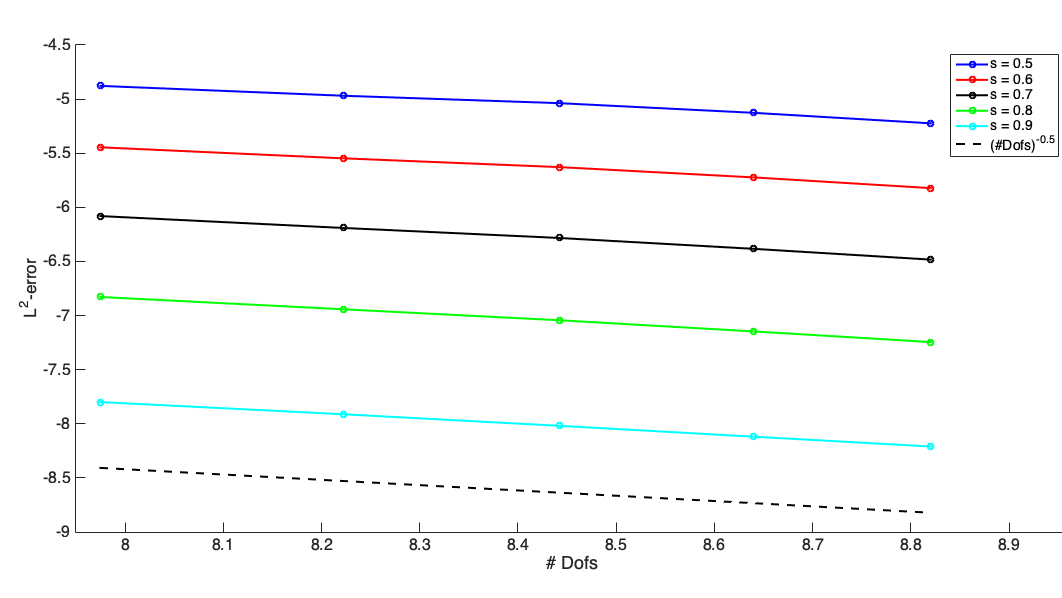}
\end{center}
\vspace{-0.5cm}
\caption{Global $L^2$-errors for the finite element solution to Example \ref{ex:nonsmooth} over quasi-uniform meshes with $s = 0.5, 0.6, 0.7, 0.8, 0.9$. The decay rate $N^{-1/2}$, which is of lower order than the interpolation error, is consistent with \eqref{eq:AN} for $s\ge 1/2$.}
\label{fig:errsL2_uniform}
\end{figure}

We next consider approximations using graded meshes that satisfy \eqref{eq:H} with $\mu=2$. By Proposition \ref{prop: global_L2_graded}, we expect a convergence rate of order $N^{-\min\{1/2+s/2,3/4\}}$, according to \eqref{eq:h-N}. In Figure \ref{fig:errsL2_graded} we display the computational rates of convergence for $s = 0.2, 0.4, 0.6, 0.8$, which are in good agreement with theory.  

\begin{figure}[htbp]
\begin{center}
\includegraphics[width=0.8\linewidth]{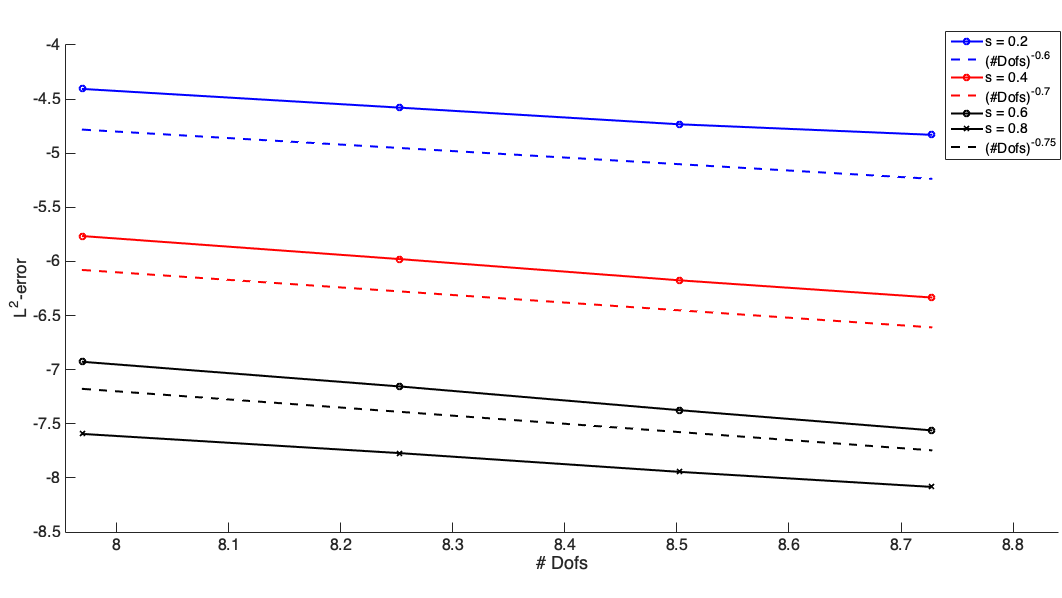} 
\end{center}
\vspace{-0.5cm}
\caption{Global $L^2$-errors for the finite element solution to Example \ref{ex:nonsmooth} over graded meshes with $\mu = 2$ and $s = 0.2, 0.4, 0.6, 0.8$. The computational decay rates are consistent with the theoretical prediction $N^{-\min\{1/2+s/2,3/4\}}$ of \eqref{eq:conv_L2_graded}.}
\label{fig:errsL2_graded}
\end{figure}

\subsection{Local $H^s$-norm error estimates}
We next explore the sharpness of our local error estimates derived in Section \ref{sec:local_estimates} and summarized in Tables \ref{tab:hresult} and \ref{tab:hresult_graded}. More precisely, we find computational rates of convergence in $H^s(B(0,0.3))$, namely the ball of radius $0.3$ centered at the origin, upon evaluating $| I_h u - u_h |_{H^s(B(0,0.3))}$ via the same techniques used when building the stiffness matrix. This is because
\[
| u - u_h |_{H^s(B(0,0.3))} \le | u - I_h u |_{H^s(B(0,0.3))} + | I_h u - u_h |_{H^s(B(0,0.3))}
\]
and the first term in the right hand side above is of higher order than the second for the locally smooth function $u$ of \eqref{eq:getoor}.
We display the errors in $H^s(B(0,0.3))$ for $s = 0.2, 0.4, 0.6, 0.8$ in Figures \ref{fig:localHs_uniform} and \ref{fig:localHs_graded} for quasi-uniform and graded meshes, respectively. We observe good agreement with the theoretical rates $N^{-\min\{\frac{1}{4}+\frac{s}{2},\frac{1}{2}\}}$ of Table  \ref{tab:hresult} and $N^{-\min\{\frac{1}{2}+\frac{s}{2},1-\frac{s}{2}\}}$ of Table \ref{tab:hresult_graded} in each case.

\begin{figure}[htbp]
\begin{center}
\includegraphics[width=0.8\linewidth]{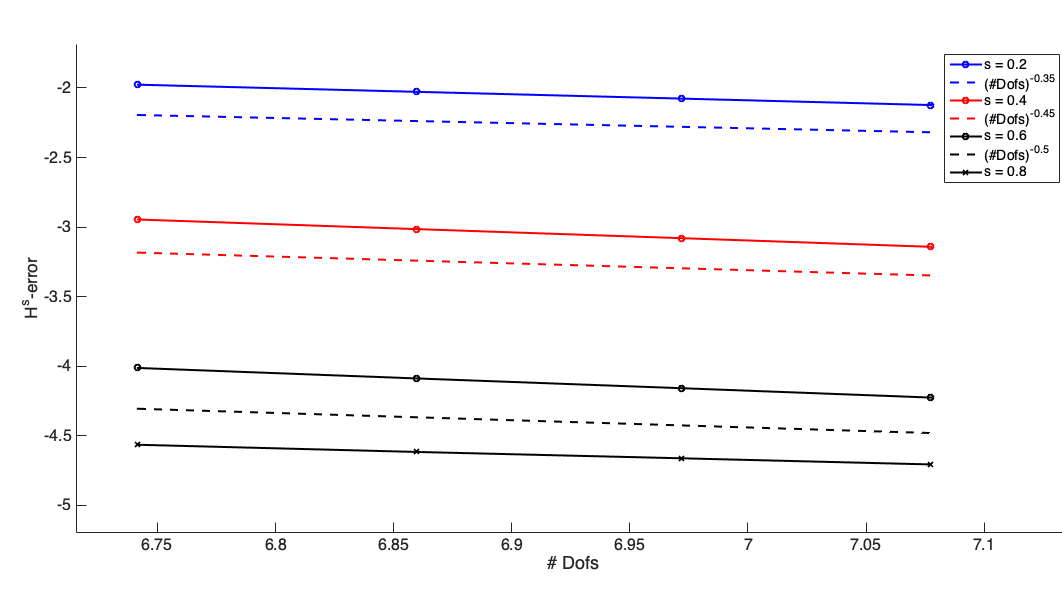}
\end{center}
\vspace{-0.5cm}
\caption{Errors in $H^s(B(0,0.3))$ for the finite element solution to Example \ref{ex:nonsmooth} over quasi-uniform meshes with $s = 0.2, 0.4, 0.6, 0.8$. Computational rates are consistent with the theoretical rates $N^{-\min\{\frac{1}{4}+\frac{s}{2},\frac{1}{2}\}}$ of Table \ref{tab:hresult}.}
\label{fig:localHs_uniform}
\end{figure}

\begin{figure}[htbp]
\begin{center}
\includegraphics[width=0.8\linewidth]{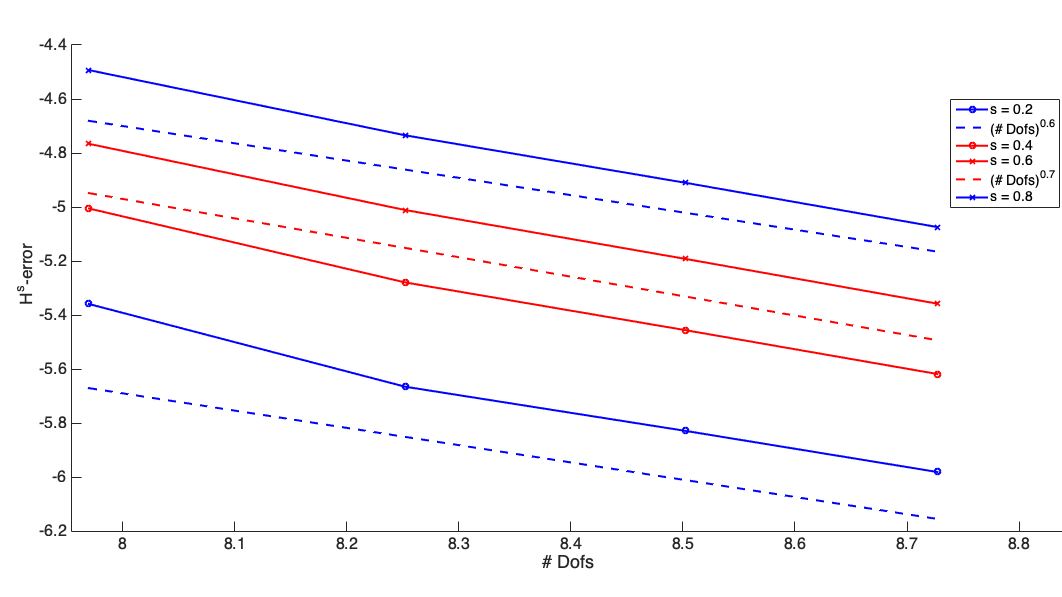} 
\end{center}
\vspace{-0.5cm}
\caption{Errors in $H^s(B(0,0.3))$ for the finite element solution to Example \ref{ex:nonsmooth} over graded uniform meshes with $\mu =2$ and $s = 0.2, 0.4, 0.6, 0.8$. Computational rates are consistent with the theoretical rates $N^{-\min\{\frac{1}{2}+\frac{s}{2},1-\frac{s}{2}\}}$ of Table \ref{tab:hresult_graded}.}
\label{fig:localHs_graded}
\end{figure}

Finally we emphasize that, according to our discussion in Section \ref{sec:energy-norm_estimates}, the {\em global} $H^s$-errors decay with rate $N^{-1/4}$ (for uniform meshes) and $N^{-\frac{1}{2}}$ (for graded meshes); see \eqref{eq:further-reg} and \eqref{eq:conv_Hs_graded}. It can be seen from our numerical experiments that in all cases the finite element solutions converge with higher order in $H^s(B(0,0.3))$. Therefore, these experiments illustrate that the finite element error is effectively concentrated around $\pp\Omega$.

\bibliography{fractional}{}

\begin{thebibliography}{10}

\bibitem{ABB}
{\sc G.~Acosta, F.~Bersetche, and J.~Borthagaray}, {\em A short {FE}
  implementation for a 2d homogeneous {D}irichlet problem of a fractional
  {L}aplacian.}, Comput. Math. Appl., 74 (2017), pp.~784--816.

\bibitem{AcosBort2017fractional}
{\sc G.~Acosta and J.~Borthagaray}, {\em A fractional {L}aplace equation:
  regularity of solutions and finite element approximations}, SIAM J. Numer.
  Anal., 55 (2017), pp.~472--495.

\bibitem{ainsworth2017aspects}
{\sc M.~Ainsworth and C.~Glusa}, {\em Aspects of an adaptive finite element
  method for the fractional laplacian: a priori and a posteriori error
  estimates, efficient implementation and multigrid solver}, Comput. Methods
  Appl. Mech. Engrg., 327 (2017), pp.~4--35.

\bibitem{Babuska:79}
{\sc I.~Babu{\v{s}}ka, R.~Kellogg, and J.~Pitk{\"a}ranta}, {\em Direct and
  inverse error estimates for finite elements with mesh refinements}, Numer.
  Math., 33 (1979), pp.~447--471.

\bibitem{BiWaZu17}
{\sc U.~Biccari, M.~Warma, and E.~Zuazua}, {\em Local elliptic regularity for
  the {D}irichlet fractional {L}aplacian}, Adv. Nonlinear Stud., 17 (2017),
  pp.~387--409, \url{https://doi.org/https://doi.org/10.1515/ans-2017-0014},
  \url{https://www.degruyter.com/view/journals/ans/17/2/article-p387.xml}.

\bibitem{BBNOS18}
{\sc A.~Bonito, J.~Borthagaray, R.~Nochetto, E.~Ot{\'{a}}rola, and A.~Salgado},
  {\em Numerical methods for fractional diffusion}, Comput. Vis. Sci., 19
  (2018), pp.~19--46, \url{https://doi.org/10.1007/s00791-018-0289-y},
  \url{https://doi.org/10.1007/s00791-018-0289-y}.

\bibitem{BoLePa17}
{\sc A.~Bonito, W.~Lei, and J.~Pasciak}, {\em Numerical approximation of the
  integral fractional {L}aplacian}, Numer. Math., 142 (2019), pp.~235--278,
  \url{https://doi.org/10.1007/s00211-019-01025-x},
  \url{https://doi.org/10.1007/s00211-019-01025-x}.

\bibitem{BoCi19}
{\sc J.~Borthagaray and P.~Ciarlet~Jr.}, {\em On the convergence in
  ${H}^1$-norm for the fractional {L}aplacian}, SIAM J. Numer. Anal., 57
  (2019), pp.~1723--1743.

\bibitem{BoLiNo19}
{\sc J.~Borthagaray, W.~Li, and R.~Nochetto}, {\em Linear and nonlinear
  fractional elliptic problems}, in 75 Years of Mathematics of Computation,
  vol.~754 of {Contemp. Math.}, Amer. Math. Soc., Providence, RI, 2020,
  pp.~69--92.

\bibitem{BoNo19}
{\sc J.~Borthagaray and R.~Nochetto}, {\em Besov regularity for the dirichlet
  integral fractional laplacian in lipschitz domains}, arXiv preprint
  arXiv:2110.02801,  (2021).

\bibitem{BoNoSa18}
{\sc J.~Borthagaray, R.~Nochetto, and A.~Salgado}, {\em Weighted sobolev
  regularity and rate of approximation of the obstacle problem for the integral
  fractional {L}aplacian}, Math. Models Methods Appl. Sci., 29 (2019),
  pp.~2679--2717.

\bibitem{Brasco:16}
{\sc L.~Brasco and E.~Parini}, {\em The second eigenvalue of the fractional
  {$p$}-{L}aplacian}, Adv. Calc. Var., 9 (2016), pp.~323--355,
  \url{https://doi.org/10.1515/acv-2015-0007},
  \url{https://doi.org/10.1515/acv-2015-0007}.

\bibitem{CiarletJr}
{\sc P.~Ciarlet, Jr.}, {\em {Analysis of the Scott-Zhang interpolation in the
  fractional order Sobolev spaces}}, J. Numer. Math., 21 (2013), pp.~173--180,
  \url{https://doi.org/10.1515/jnum-2013-0007},
  \url{http://hal.inria.fr/hal-00937677}.

\bibitem{CozziM_2017}
{\sc M.~Cozzi}, {\em Interior regularity of solutions of non-local equations in
  {S}obolev and {N}ikol'skii spaces}, Ann. Mat. Pura Appl. (4), 196 (2017),
  pp.~555--578, \url{https://doi.org/10.1007/s10231-016-0586-3},
  \url{https://doi.org/10.1007/s10231-016-0586-3}.

\bibitem{DemlowA_GuzmanJ_SchatzAH_2011}
{\sc A.~Demlow, J.~Guzm{\'a}n, and A.~Schatz}, {\em Local energy estimates for
  the finite element method on sharply varying grids}, Math. Comp., 80 (2011),
  pp.~1--9, \url{https://doi.org/10.1090/S0025-5718-2010-02353-1},
  \url{http://dx.doi.org/10.1090/S0025-5718-2010-02353-1}.

\bibitem{DiCastro:16}
{\sc A.~Di~Castro, T.~Kuusi, and G.~Palatucci}, {\em Local behavior of
  fractional {$p$}-minimizers}, Ann. Inst. H. Poincar\'{e} Anal. Non
  Lin\'{e}aire, 33 (2016), pp.~1279--1299,
  \url{https://doi.org/10.1016/j.anihpc.2015.04.003},
  \url{https://doi.org/10.1016/j.anihpc.2015.04.003}.

\bibitem{duo2018novel}
{\sc S.~Duo, H.~van Wyk, and Y.~Zhang}, {\em A novel and accurate finite
  difference method for the fractional {L}aplacian and the fractional poisson
  problem}, J. Comput. Phys., 355 (2018), pp.~233--252.

\bibitem{duo2019accurate}
{\sc S.~Duo and Y.~Zhang}, {\em Accurate numerical methods for two and three
  dimensional integral fractional {L}aplacian with applications}, Comput.
  Methods Appl. Mech. Engrg., 355 (2019), pp.~639--662.

\bibitem{Faermann2}
{\sc B.~Faermann}, {\em Localization of the {A}ronszajn-{S}lobodeckij norm and
  application to adaptive boundary element methods. {I}. {T}he two-dimensional
  case}, IMA J. Numer. Anal., 20 (2000), pp.~203--234,
  \url{https://doi.org/10.1093/imanum/20.2.203},
  \url{http://dx.doi.org/10.1093/imanum/20.2.203}.

\bibitem{Faermann}
{\sc B.~Faermann}, {\em Localization of the {A}ronszajn-{S}lobodeckij norm and
  application to adaptive boundary element methods. {II}. {T}he
  three-dimensional case}, Numer. Math., 92 (2002), pp.~467--499,
  \url{https://doi.org/10.1007/s002110100319},
  \url{http://dx.doi.org/10.1007/s002110100319}.

\bibitem{faustmann2020local}
{\sc M.~Faustmann, M.~Karkulik, and J.~Melenk}, {\em {Local convergence of the
  FEM for the integral fractional Laplacian}}, arXiv preprint arXiv:2005.14109,
   (2020).

\bibitem{faustmann2019quasi}
{\sc M.~Faustmann, J.~Melenk, and D.~Praetorius}, {\em Quasi-optimal
  convergence rate for an adaptive method for the integral fractional
  {L}aplacian}, arXiv preprint arXiv:1903.10409,  (2019).

\bibitem{Getoor}
{\sc R.~Getoor}, {\em First passage times for symmetric stable processes in
  space}, Trans. Amer. Math. Soc., 101 (1961), pp.~75--90.

\bibitem{Gimperlein:19}
{\sc H.~Gimperlein, E.~Stephan, and J.~Stocek}, {\em Corner singularities for
  the fractional laplacian and finite element approximation}.
\newblock Preprint available at
  {\url{http://www.macs.hw.ac.uk/~hg94/corners.pdf}}, 2019.

\bibitem{gimperlein2019space}
{\sc H.~Gimperlein and J.~Stocek}, {\em Space--time adaptive finite elements
  for nonlocal parabolic variational inequalities}, Comput. Methods Appl. Mech.
  Engrg., 352 (2019), pp.~137--171.

\bibitem{Grisvard}
{\sc P.~Grisvard}, {\em Elliptic problems in nonsmooth domains}, vol.~24 of
  Monographs and Studies in Mathematics, Pitman (Advanced Publishing Program),
  Boston, MA, 1985.

\bibitem{Grubb}
{\sc G.~Grubb}, {\em Fractional {L}aplacians on domains, a development of
  {H}\"ormander's theory of $\mu$-transmission pseudodifferential operators},
  Adv. Math., 268 (2015), pp.~478--528,
  \url{https://doi.org/http://dx.doi.org/10.1016/j.aim.2014.09.018},
  \url{http://www.sciencedirect.com/science/article/pii/S0001870814003302}.

\bibitem{huang2014numerical}
{\sc Y.~Huang and A.~Oberman}, {\em Numerical methods for the fractional
  {L}aplacian: A finite difference-quadrature approach}, SIAM J. Numer. Anal.,
  52 (2014), pp.~3056--3084.

\bibitem{Kuusi:15}
{\sc T.~Kuusi, G.~Mingione, and Y.~Sire}, {\em Nonlocal self-improving
  properties}, Anal. PDE, 8 (2015), pp.~57--114,
  \url{https://doi.org/10.2140/apde.2015.8.57},
  \url{https://doi.org/10.2140/apde.2015.8.57}.

\bibitem{lin2014lower}
{\sc Q.~Lin, H.~Xie, and J.~Xu}, {\em Lower bounds of the discretization error
  for piecewise polynomials}, Math. Comp., 83 (2014), pp.~1--13.

\bibitem{mclean2000strongly}
{\sc W.~McLean}, {\em Strongly elliptic systems and boundary integral
  equations}, Cambridge university press, 2000.

\bibitem{NitscheJA_SchatzAH_1974a}
{\sc J.~Nitsche and A.~Schatz}, {\em Interior estimates for {R}itz-{G}alerkin
  methods}, Math. Comp., 28 (1974), pp.~937--958.

\bibitem{NochettoRH_PaoliniM_VerdiC_1991}
{\sc R.~Nochetto, M.~Paolini, and C.~Verdi}, {\em An adaptive finite element
  method for two-phase {S}tefan problems in two space dimensions. {I}.
  {S}tability and error estimates}, Math. Comp., 57 (1991), pp.~73--108,
  S1--S11, \url{https://doi.org/10.2307/2938664},
  \url{https://doi.org/10.2307/2938664}.

\bibitem{nochetto2010posteriori}
{\sc R.~Nochetto, T.~von Petersdorff, and C.-S. Zhang}, {\em A posteriori error
  analysis for a class of integral equations and variational inequalities},
  Numer. Math., 116 (2010), pp.~519--552.

\bibitem{RosOtonSerra}
{\sc X.~Ros-Oton and J.~Serra}, {\em The {D}irichlet problem for the fractional
  {L}aplacian: regularity up to the boundary}, J. Math. Pures Appl., 101
  (2014), pp.~275--302,
  \url{https://doi.org/http://dx.doi.org/10.1016/j.matpur.2013.06.003},
  \url{http://www.sciencedirect.com/science/article/pii/S0021782413000895}.

\bibitem{Savare98}
{\sc G.~Savar{\'e}}, {\em Regularity results for elliptic equations in
  {L}ipschitz domains}, J. Funct. Anal., 152 (1998), pp.~176--201.

\bibitem{VishikEskin}
{\sc M.~I. Vi{\v{s}}ik and G.~I. {\`E}skin}, {\em Convolution equations in a
  bounded region}, Uspehi Mat. Nauk, 20 (1965), pp.~89--152.
\newblock English translation in {\em Russian Math. Surveys}, 20:86-151, 1965.

\bibitem{zhao2017adaptive}
{\sc X.~Zhao, X.~Hu, W.~Cai, and G.~Karniadakis}, {\em Adaptive finite element
  method for fractional differential equations using hierarchical matrices},
  Comput. Methods Appl. Mech. Engrg., 325 (2017), pp.~56--76.

\end{thebibliography}
\bibliographystyle{siamplain}
\end{document}